\documentclass[10pt]{amsart}
\usepackage{a4wide}
\usepackage{amsmath, amsthm, amsfonts, amscd, amssymb, color}
\usepackage[english]{babel}
\usepackage[utf8x]{inputenc}
\usepackage{caption}
\usepackage{setspace}
\usepackage{epstopdf}
\usepackage{graphicx}

\usepackage{xcolor}

\newtheorem{theorem}{Theorem}
\newtheorem{definition}{Definition}
\newtheorem{proposition}{Proposition}
\newtheorem{lemma}{Lemma}
\newtheorem{corollary}{Corollary}
\newtheorem{remark}{Remark}

\newcommand{\p}{\mathbb{P}}
\newcommand{\e}{\mathbb{E}}
\newcommand{\reals}{\mathbb{R}}
\newcommand{\ind}{\mathbf{1}}
\newcommand{\Wq}{\psi}
\newcommand{\Wqprime}{\psi^{\prime}}

\newcommand{\Zq}{\Psi}
\newcommand{\Zqprime}{\Psi^{\prime}}
\newcommand{\Hq}{\varphi}
\newcommand{\Hqprime}{\varphi^{\prime}}

\newcommand{\uGell}{\underline{G}^\ell}

\newcommand{\dd}{\mathrm{d}}

\begin{document}

\title[An optimization dichotomy]{An optimization dichotomy for capital injections and absolutely continuous dividend strategies}

\author[Renaud]{Jean-Fran\c{c}ois Renaud}
\address{D\'epartement de math\'ematiques, Universit\'e du Qu\'ebec \`a Montr\'eal (UQAM), 201 av.\ Pr\'esident-Kennedy, Montr\'eal (Qu\'ebec) H2X 3Y7, Canada}
\email{renaud.jf@uqam.ca}

\author[Roch]{Alexandre Roch}
\address{D\'epartement de finance, Universit\'e du Qu\'ebec \`a Montr\'eal (UQAM),  315  Sainte-Catherine East, Montr\'eal, (Qu\'ebec) H2X 3X2, Canada}
\email{roch.alexandre\_f@uqam.ca}

\author[Simard]{Clarence Simard}
\address{D\'epartement de math\'ematiques, Universit\'e du Qu\'ebec \`a Montr\'eal (UQAM), 201 av.\ Pr\'esident-Kennedy, Montr\'eal (Qu\'ebec) H2X 3Y7, Canada}
\email{simard.clarence@uqam.ca}


\date{\today}


\begin{abstract}
We consider an optimal stochastic control problem in which a firm's cash/surplus process is controlled by dividend payments and capital injections. Stockholders aim to maximize their dividend stream minus the cost of injecting capital, if needed. We consider absolutely continuous dividend policies subject to a level-dependent upper bound on the dividend rate while we allow for general capital injections behavior. We prove that the optimal strategy can only be of two types: dividends are paid according to a \textit{mean-reverting} strategy with capital injections performed each time the cash process reaches zero; or, dividends are paid according to another \textit{mean-reverting} strategy and no injection of capital is ever made, until ruin is reached. We give a complete solution to this problem and characterize this dichotomy by comparing (the derivatives of) the value functions at zero of two sub-problems. The first sub-problem is concerned solely with the maximization of dividends, while the second sub-problem is the corresponding bail-out optimal dividend problem for which we provide also a complete solution.
\end{abstract}

\maketitle


\section{Introduction}

The problem of maximizing a stream of dividend payments under a ruin constraint goes back to Bruno de Finetti \cite{definetti_1957} in the context of insurance mathematics and Jeanblanc \& Shiryaev \cite{jeanblanc-shiryaev_1995} in a Brownian setup.  In this paper, we consider absolutely continuous dividend policies and we allow general capital injection controls. We give a complete solution to the problem of optimizing simultaneously dividends and capital injections when the value of the firm follows an arithmetic Brownian motion and the rate of dividends is bounded by an increasing concave function of the controlled process. The optimal pair depends on the parameters and the upper bound of the dividend rate according to the following dichotomy: dividends are paid using a bang-bang strategy while capital injections must be made each time the cash/surplus process reaches zero, so the firm is never ruined; or, dividends are paid using another bang-bang strategy and no injection of capital is ever made until ruin. No in-between strategy is optimal. If it is optimal to inject capital, then those injections will take the form of bail-out payments, i.e., they will occur only when the cash process hits zero, making the controlled process a reflected (at zero) process. On the other hand, the dividend component of the optimal strategy always consists in using the maximum dividend rate allowed when the controlled process is above an optimal threshold, giving rise to a refracted process in between capital injections (or up to the time of ruin). Our methods of proof rely on a combination of probabilistic and analytical techniques, namely fluctuation identities for diffusions and the theory of viscosity solutions for the associated Hamilton-Jacobi-Bellman (HJB) equations.

\subsection{Problem formulation}

On a filtered probability space $\left( \Omega, \mathcal{F}, \mathbb{F} =  \left\lbrace \mathcal{F}_t, t \geq 0\right\rbrace, \p \right)$, consider a Brownian motion with drift $X=\left\lbrace X_t , t \geq 0 \right\rbrace$ given by
\begin{equation}\label{eq:ABM}
 \dd X_t = \mu  \dd t + \sigma  \dd B_t ,
\end{equation}
where $\mu \in \reals$ and $\sigma > 0$, and where $B=\left\lbrace B_t , t \geq 0 \right\rbrace$ is a standard Brownian motion.

A control strategy is a pair $\pi = (\ell,G)$ of adapted processes $\ell=\left\lbrace \ell_t , t \geq 0 \right\rbrace$ and $G=\left\lbrace G_t , t \geq 0 \right\rbrace$, where $\ell_t$ is the dividend rate at time $t$ and $G_t$ is the amount of capital injections made up to time $t$. It is assumed that $G$ is nondecreasing and right-continuous with $G_{0-}=0$. The corresponding controlled process $X^\pi$ is given by
\begin{equation*}
 \dd X^{\pi}_t = \left( \mu - \ell_t \right)  \dd t + \sigma  \dd B_t +  \dd G_t ,
\end{equation*}
with $X^{\pi}_{0-} = X_0$. We also define the termination time by $T^\pi = \inf \{ t>0 : X^{\pi}_t < 0 \}$. In the spirit of \cite{RS2021} and \cite{LR2023}, we will impose a state-dependent bound on the dividend rate. More specifically, let $F \colon \reals_+ \to \reals_+$ be a nondecreasing, concave and continuously differentiable function such that $F(0) \geq 0$. A strategy $\pi = (\ell, G )$ is admissible if, for all $t \geq 0$, we have
\[
0 \leq \ell_t \leq F(X^{\pi}_t) .
\]
The set of admissible strategies will be denoted by $\Pi$. Also, we  use the Markovian notation $\e_x$ to denote the expectation given that $X^{\pi}_{0-} = X_0 = x$, for a fixed $x \geq 0$.

We aim at solving the following control problem. Let $q > 0$ be the discounting rate and $\beta>1$ the cost of injecting capital. For $x \geq 0$, we define the performance function of a strategy $\pi = (\ell, G) \in \Pi$ by
\[ J(x;\pi) = \e_x \left[\int_0^{T^\pi} \mathrm{e}^{-qt} (\ell_t  \dd  t - \beta  \dd G_t) \right].\]
We want to compute the value function defined by
\begin{equation}\label{eq:value_fct}
V(x) = \sup_{\pi \in \Pi} J(x;\pi), \quad x \geq 0,
\end{equation}
and find an optimal strategy, i.e., an admissible strategy for which the supremum is attained.

\subsection{Related literature, contributions and organization of the paper}

In the literature, optimal dividend strategies found are often of the bang-bang type. For instance, in the absolutely continuous problem, the optimal policy often consists in paying at the maximum (constant) rate allowed when the cash process reaches a certain threshold and nothing otherwise. Similarly, in the singular case, the optimal behavior yields a reflected (at an optimal barrier) controlled process and no dividend payments below \cite{AG2008,ASW2011,jeanblanc-shiryaev_1995}.  We follow a new trend of dividend maximization problems aiming for more realistic optimal dividend strategies by imposing a level-dependent bound on the dividend rate in the absolutely continuous case \cite{RS2021,LR2023}. For instance, in the linear-bound case of \cite{RS2021}, dividend rates are bounded by a \textit{proportion} of the current level of cash, allowing for the more natural situation in which the dividend rate increases as a function of the financial health of the firm.

Maximization of dividends with or without bail-out payments has been studied extensively in a variety of models. The expression ``bail-out payments" usually refers to forced injections of capital when the cash process reaches (or goes below) zero, making the controlled process a reflected process. Sethi and Taksar \cite{ST2002} consider a model in which dividends are paid as singular (reflective) controls and the firm is not allowed to go bankrupt and must issue external equity to keep the firm afloat.  Shreve et al. \cite{shreve-lehoczky-gaver_1984} present a similar mathematical model in the context of inventory control. The optimal dichotomy between injecting capital or letting the cash process get ruined has been first studied by L{\o}kka and Zervos \cite{LZ2008} (a result sometimes called the ``L{\o}kka-Zervos alternative"). Avram et al. \cite{AGR2019,AGLW2021} consider a similar dichotomy for Cram\'er-Lundberg models. This control problem in a singular setting  has also been studied in a L\'evy model with upward jumps by Avanzi et al. \cite{ASW2011}. Adding upward jumps to a Brownian motion does not significantly alter the difficulty of the problem, but it becomes more difficult when there are downward jumps \cite{AGLW2021}.

The theory of viscosity solutions for variational inequalities has been used extensively in the literature to study stochastic control problems. However, the absence of explicit expressions for the optimal performance function and the optimal strategy often forbids a thorough study of the optimal controlled process. In control problems related to the payment of dividends, a number of authors have tackled the question by proposing a large number of financial models and used viscosity theory to give a solution. Choulli et al. \cite{CTZ2003} describe the optimal dividend and business policy in a diffusion model using a surplus process in which dividends are paid in a reflective way and the optimal business risk can be described in terms of optimal thresholds.  D\'ecamps and Villeneuve \cite{DV2007} consider a dividend control problem of the singular type (reflective) with a growth option (modelled by a stopping time) giving rise to a dichotomy in which one possible solution involves exercising the growth option when the surplus reaches a certain threshold and the other solution is to entirely ignore the option. Similarly, Ly Vath et  al.  \cite{LVPV2008} present a dividend model with reversible technology investments. Chevalier et al. \cite{CLVR2020} present a dividend and capital injection model in which a grace period is allowed before the firm is declared bankrupt. Using viscosity comparison arguments, it is shown that the optimal capital injection strategy occurs as a reflection when the surplus hits zero after some time spent in the grace period and as an impulse when the grace period expires.  More recently, Albrecher et al. \cite{AAM2020,AAM2022} proposed a dividend model in which the dividend rate is not allowed to decrease over time (so-called ratcheting strategies) and use the theory of viscosity solutions to characterize the value function in terms of its HJB equation and provide numerical examples.

Commonly, the main difficulty in optimal dividend problems (for both the refracted and reflected types) is to show the existence of an optimal threshold, corresponding to an optimal barrier strategy, and characterize its value. Fluctuation theory, when applied to stochastic control problems, often leads to explicit (or semi-explicit) expressions for performance functions associated to barrier-type strategies. The optimal performance function among this subset of strategies is then computed by using probabilistic arguments together with fluctuation identities. If the subset of strategies contains the value function of the initial problem and if one can identify it, then the optimality can then be shown using a verification lemma, as long as it is smooth enough.

In this work, we solve two stochastic control problems: a (forced) bail-out optimal dividends problem and the main optimization problem in which capital injections are not forced. For the bail-out problem, we combine ideas from the fluctuation theory approach and the viscosity theory approach to fully characterize the solution. In particular, we are able to show the existence of the optimal barrier and characterize it. This allows us to consider a general set of admissible strategies while maintaining the advantage of HJB equations analysis (which involve derivatives of the value function) to describe the behavior of optimal strategies, even though smoothness of the value function is not known a priori.  Therefore, one of our main contributions is methodological. In the bail-out optimal dividends problem, a strategy is one-dimensional as it consists of a dividend payment rate, while capital injections behaviors is pre-determined. In the main optimization problem,  we consider two-dimensional strategies consisting of a capital injection policy and a dividend payment rate. In the literature, it is frequently argued that optimal injections should occur only at zero, involving an inspection of the associated variational inequality and/or using analytical properties of the performance function in combination with verification theorems. As a first step to solve our general control problem, we directly show that among the large class of admissible capital injection strategies (i.e., adapted and nondecreasing processes) the only potentially optimal behaviors are continuous and can only increase at zero. This is done by comparing solutions of stochastic differential equations (SDEs) controlled by a nondecreasing process. In other words, the optimal injections make the controlled process a reflected (at zero) process. As such, optimal injections can be seen as a last resort to keep the cash process from ruin, when it is optimal to do so. Then, the theory of viscosity solutions also allows us to derive a dichotomy that it is optimal to either always inject at zero or never inject, with the help of a comparison principle for solutions of the associated HJB equation. This new comparison result is based on either the initial values of the super and subsolutions or their initial derivatives. 

The rest of the paper is organized as follows. In the next section, we present the main results of this paper. In Section~\ref{sec:notation}, notation and preliminary results, needed in the rest of the paper, are defined and given. In Section~\ref{sec:Vd}, the solution to the problem without injections is presented, while in Section~\ref{sec:Vc}, we solve the bail-out dividend problem, i.e., the problem with forced injections. Finally, in Section~\ref{sec:dichotomy}, a solution to the main problem is provided.

\section{Main results}\label{sec:main-results}

In this section, we present the solutions of three optimization problems, of which two are new.

\subsection{Maximization of dividends without capital injections}

First, we consider the maximization of dividends when capital injections are not allowed. To wit, define
\[
\Pi_d = \left\lbrace \pi = (\ell, G) \in \Pi \colon G \equiv 0 \right\rbrace ,
\]
with corresponding  value function
\[
V_d(x) = \sup_{\pi \in \Pi_d}  J(x;\pi), \quad x \geq 0.
\]
When $F(x)=S$ for a given constant $S>0$,  we recover the classical problem studied by \cite{jeanblanc-shiryaev_1995}. If $F(x)=Kx$, for a given constant $K>0$, then we recover the control problem studied by \cite{RS2021}. Finally, the general case, i.e., when $F$ is a function with the abovementioned properties, the problem has been solved  recently by \cite{LR2023}. To state their result, let $b\geq0$ and define the mean-reverting dividend strategy at level $b$ by $\ell^b_t = F(X^b_t) \mathbf{1}_{\{X^b_t \geq b \}}$ with the associated controlled process
\[
 \dd X^b_t = \left(\mu - F(X^b_t) \mathbf{1}_{\{X^b_t \geq b \}} \right)  \dd t + \sigma  \dd  B_t,
\]
and with termination time $T^b = \inf \{ t>0 : X^b_t <0 \}$. Existence of a strong solution to this stochastic differential equation (SDE) is given in \cite{LS2016}.

\begin{theorem}[\cite{LR2023}] \label{thm:Vd}
There exists a constant $b_d \geq 0$ such that a mean-reverting dividend strategy at level $b_d$ is optimal for the control problem without capital injections. More precisely, for all $x \geq 0$, we have
\[
V_d(x) = \e_x \left[\int_0^{T^{b_d}} \mathrm{e}^{-qs} F(X^{b_d}_s) \mathbf{1}_{\{X^{b_d}_s \geq b_d \}}  \dd s \right].
\]
\end{theorem}

The optimal barrier level $b_d$ is a solution of Equation~\eqref{eq.bd} in Proposition~\ref{prop.bd}. Also, an analytical expression for $V_d$ is available. Of course, we have $V_d(0)=0$. More details in Section~\ref{sec:Vd}.

\begin{remark}
The term "\textit{mean-reverting strategy}" originates from \cite{AW2012}. Since then, it has been used also by the follow-up literature, even when it is not standard \textit{mean-reversion}. We also adopt this terminology.
\end{remark}

\subsection{Maximization of dividends with forced capital injections}

Second, we consider the maximization of dividends with forced capital injections, i.e., preventing the controlled process from going below zero. Specifically, define
\[
\Pi_c = \left\lbrace \pi \in \Pi \colon X^\pi_t \geq 0 \mbox{ for all } t\geq 0 \right\rbrace
\]
with corresponding value function
\[
V_c(x) = \sup_{\pi \in \Pi_c} J(x;\pi), \quad x \geq 0.
\]
Note that, for each $\pi \in \Pi_c$, we have $T^\pi=\infty$.

In order to present our solution to this optimization problem, we introduce controlled processes reflected at zero. Fix an admissible dividend rate $\ell$. For $t \geq 0$, define $\uGell_t = - \min \{ X^{(\ell,0)}_s \wedge 0: 0\leq s \leq t \}$.
In this case, $\pi = (\ell, \uGell) \in \Pi_c$ and the corresponding controlled process is nonnegative at all times, i.e., the process given by
\[
 \dd X^\pi_t = (\mu - \ell_t) \dd t + \sigma  \dd B_t +  \dd  \uGell_t ,
\]
with $X^\pi_0 \geq 0$, is such that $X^\pi_t \geq 0$ for all $t \geq 0$ and
\[
\uGell_t = \int_0^t \mathbf{1}_{\{X^{\pi}_s = 0\}}  \dd \uGell_s .
\]
In other words, the process $X^\pi$ is reflected at zero. This is the famous Skorohod problem. See, for instance, Lemma 3.6.14 in \cite{karatzas-shreve_1991}.

Extending the notation from the previous section, for a given $b\geq 0$, we define the reflected mean-reverting dividend strategy at level $b$ by the dividend rate $\underline{\ell}^b_t = F(\underline{X}^b_t) \mathbf{1}_{\{\underline{X}^b_t \geq b \}}$, with associated controlled reflected process
\[
 \dd  \underline{X}^b_t = \left(\mu - F(\underline{X}^b_t) \mathbf{1}_{\{\underline{X}^b_t \geq b \}} \right)  \dd t + \sigma  \dd  B_t +  \dd  \underline{G}^{b}_t,
\]
where the injection process $\underline{G}^{b}$ satisfies
\[
\underline{G}^{b}_t = \int_0^t \mathbf{1}_{\{\underline{X}^b_s = 0 \}}  \dd  \underline{G}^{b}_s .
\]
Of course, we have $\underline{X}^b_t \geq 0$ for all $t \geq 0$. The existence of a solution to this specific Skorohod problem can be found in Lemma~\ref{lemma.Skorohod} in Appendix \ref{AppendixA}. Note that the representation $\underline{G}^{b} = \underline{G}^{\ell^b}$ is valid, but can only be well defined once $\underline{X}^b$ is shown to exist.

A solution of this second maximization problem is given in the following theorem and its proof is the objective of Section~\ref{sec:Vc}.

\begin{theorem}\label{thm:Vc}
There exists a  constant $b_c > 0$ such that a reflected mean-reverting dividend strategy at level $b_c$ is optimal for the control problem with forced capital injections. More precisely, for all $x \geq 0$, we have
\[
V_c(x) = \e_x \left[ \int_0^\infty \mathrm{e}^{-qs} \left( F(\underline{X}^{b_c}_s) \mathbf{1}_{\{\underline{X}^{b_c}_s \geq b_c \}}  \dd s - \beta \dd \underline{G}^{b_c}_s \right) \right].
\]
\end{theorem}

The optimal barrier level $b_c$ is characterized as the unique solution of either Equation~\eqref{eq.bc1} or Equation~\eqref{eq.bc2} given below. An analytical expression for $V_c$ is also derived in Section~\ref{sec:Vc}. In particular, we will show that $V'_c(0+)=\beta$.

\subsection{Optimization of capital injections and dividend payments}

Using the solutions of the previous two maximization problems, our solution of the main optimization problem~\eqref{eq:value_fct} is as follows.
\begin{theorem}\label{th.dichotomy}
The value function $V$ is characterized by the following dichotomy:
\begin{itemize}
\item if $V_d^\prime(0+) < \beta$, then an optimal strategy is given by $(\ell^{b_d},0)$ and $V = V_d$;
\item if $V_d^\prime(0+) > \beta$, then an optimal strategy is given by $(\underline{\ell}^{b_c}, \underline{G}^{b_c})$ and $V = V_c$.
\end{itemize}
Equivalently,
\begin{itemize}
\item if $V_c(0) < 0 $, then an optimal strategy is given by $(\ell^{b_d},0)$ and $V = V_d$;
\item if $V_c(0) > 0 $, then an optimal strategy is given by $(\underline{\ell}^{b_c}, \underline{G}^{b_c})$ and $V = V_c$.
\end{itemize}
Finally, if $V_d^\prime(0+) = \beta$ or $V_c(0)=0$, then both $(\ell^{b_d},0)$ and $(\underline{\ell}^{b_c}, \underline{G}^{b_c})$ are optimal strategies, and we have $V = V_c = V_d$.
\end{theorem}

Intuitively, the condition $V_d'(0+) > \beta=V_c'(0+)$ means that, when the process reaches zero, then it is worth injecting capital and keep on going for more dividends later. On the other hand, the condition $V_d'(0+) < \beta$ means that the marginal benefit of keeping the process \textit{alive} is less than the marginal cost of injecting capital.

Similarly, if $V_c(0) < 0=V_d(0)$, then the value of keeping the process nonnegative with injections is less than the value of letting it get ruined when reaching zero. For the condition $V_c(0)>0$, we have the opposite interpretation.

Note that, if $V_d'(0+) = \beta$ or $V_c(0)=0$, then the firm is indifferent between paying dividends at level $b_c$ and injecting capital each time zero is attained, or paying dividends at level $b_d$ and never inject capital (which means getting ruined at the first passage at level zero).

\section{Notation and preliminaries}\label{sec:notation}

In this short section, we introduce notation and functions that will be used throughout the paper.

First, recall that $X_t = X_0 + \mu t + \sigma B_t$ and define $Y_t = X_t + \underline{G}_t$ where $\underline{G}_t:=- \min \{ X_s \wedge 0 : 0\leq s \leq t \}$. In words, $Y=\left\lbrace Y_t , t \geq 0 \right\rbrace$ is a Brownian motion with drift, reflected at zero.

Second, let us define first-passage stopping times. For a given process $Z=\left\lbrace Z_t , t \geq 0 \right\rbrace$ and a level $b \geq 0$, set
\[
\tau^Z_b = \inf \left\lbrace t > 0 \colon Z_t = b \right\rbrace .
\]
In what follows, the process $Z$ will be for example $X$, $Y$, $X^b$, $\underline{X}^b$,  etc.

Third, we must define several functions needed to state various probabilistic identities. For $x,b \geq 0$, set
\begin{align*}
\Wq(x) &= \frac{\mathrm{e}^{\alpha_1 x}}{\Delta} -\frac{\mathrm{e}^{\alpha_2 x}}{\Delta} ,\\
\overline{\Wq}(x) &= \int_0^x \Wq(y)  \dd y = \frac{\mathrm{e}^{\alpha_1 x}}{\alpha_1 \Delta} - \frac{\mathrm{e}^{\alpha_2 x}}{\alpha_2 \Delta} - \frac{1}{q},\\
\Zq(x) &= 1 + q \overline{\Wq}(x) = \frac{q \mathrm{e}^{\alpha_1 x}}{\alpha_1 \Delta} - \frac{q\mathrm{e}^{\alpha_2 x}}{\alpha_2 \Delta} ,\\
\overline{\Zq}(x) &= \int_0^x \Zq(y)  \dd y = \frac{q \mathrm{e}^{\alpha_1 x}}{\alpha_1^2 \Delta} - \frac{q\mathrm{e}^{\alpha_2 x}}{\alpha_2^2 \Delta} - \frac{\mu}{q} ,\\
u_{0,b}(x) &= \frac{\Zq(x)}{\Zq(b)} \left(\overline{\Zq}(b)+\frac{\mu}{q} \right) - \left(\overline{\Zq}(x) + \frac{\mu}{q} \right) ,
\end{align*}
where $\alpha_1=\sigma^{-2} \left(-\mu + \Delta \right)$, $\alpha_2=\sigma^{-2} \left(-\mu - \Delta \right)$ and $\Delta=\sqrt{\mu^2+2q \sigma^2}$. Note that $\Wq(0)=\overline{\Wq}(0)=\overline{\Zq}(0)=u_{0,b}(b)=0$ and $\Zq(0)=1$. Moreover, $u'_{0,b}(0) = -1, \overline{\Zq}'(0) = 1,$  $\Zq'(0)=\overline{\Wq}'(0) = 0$ and $\Wqprime(0) = \tfrac{2}{\sigma^2}.$

As $\alpha_1$ and $\alpha_2$ are the positive and negative roots of $\frac{\sigma^2}{2} \alpha^2 + \mu \alpha - q = 0$, respectively, it is easy to verify that $\Wq, \Zq$ and $\overline{\Zq}+\mu/q$ are solutions to
\begin{equation}\label{eq:ODE_hom}
\frac{\sigma^2}{2} v''(x) + \mu v'(x) - q v(x) = 0 ,  \quad x > 0,
\end{equation}
as linear combinations of $x \mapsto \mathrm{e}^{\alpha_1 x}$ and $x \mapsto \mathrm{e}^{\alpha_2 x}$. The function $u_{0,b}$ is also a solution to~\eqref{eq:ODE_hom} as it is a linear combination of $\Zq$ and $\overline{\Zq}+\mu/q$.

It is known that, for $0 \leq x \leq b$,
\[
\e_x \left[ \mathrm{e}^{-q \tau^Y_b} \right] = \frac{\Zq(x)}{\Zq(b)},  \e_x \left[ \mathrm{e}^{-q \tau_b^X} \ind_{\{\tau_b^X<\tau_0^X\}} \right] = \frac{\Wq(x)}{\Wq(b)}\text{ and } \e_x \left[ \int_0^{\tau^Y_b} \mathrm{e}^{-qt}  \dd  \underline{G}_t \right] = u_{0,b}(x).
\]
See, e.g., \cite{APP2007}.

By Lemma 2.4 in \cite{ekstrom-lindensjo_2021}, there exists  a unique decreasing and strictly convex solution to
\begin{equation}\label{eq:ODE_F}
\frac{\sigma^2}{2} \varphi''(x) + (\mu - F(x)) \varphi'(x) - q \varphi(x) = 0 ,  \quad x > 0,
\end{equation}
such that $\varphi(0)=1$ and $\lim_{x \to \infty} \varphi(x)=0$. Furthermore,  $\e_x \left[ \mathrm{e}^{-q \tau_b^{X^0}} \right] = \frac{\Hq(x)}{\Hq(b)}$, for $0 \leq x \leq b$.

Denote the set of continuous functions with linear growth at infinity by
\[
\mathcal W = \left\lbrace u : [0,\infty) \to \mathbb{R} \colon u \mbox{ continuous and } \sup_{x \geq 0 } \frac{|u(x)|}{1+x} < \infty \right\rbrace .
\]
For $x \geq 0$, define
\[
I_F(x) = \e_x \left[\int_0^\infty \mathrm{e}^{-qt} F(\underline{X}^0_t)  \dd t \right] .
\]
It can be shown that $I_F : \reals_+ \to \reals_+$ is the unique solution in $\mathcal{W}$ to the following ordinary differential equation (ODE):
\begin{equation} \label{ODE.nonhom}
\frac{\sigma^2}{2} u''(x) + (\mu - F(x))u'(x) - q u(x) + F(x) = 0 ,  \quad x > 0 ,
\end{equation}
with initial condition $u'(0+) = 0$.

Direct calculations show that the function $I_F^\lambda := I_F + \lambda \varphi$ is also a solution to the nonhomogeneous ODE in~\eqref{ODE.nonhom}, for any $\lambda \in \reals$.

\begin{remark}
Consider the function $\tilde{I}_F : \reals \to \reals$ introduced in \cite{LR2023} and defined by
\begin{equation}\label{eq:Itilde}
\tilde{I}_F(x) = \e_x \left[ \int_0^\infty \mathrm{e}^{-qs} F(X^{0}_s) \dd s \right].
\end{equation}
Using Markovian arguments, continuity of sample paths and the fact that $X^0$ and $\underline{X}^0$ have the same distribution up to $\tau_0^{X^0}$, with respect to $\p_x$ (for $x\geq0$), we readily see that
\begin{eqnarray*}
  \tilde{I}_F(x) &=& \e_x \left[ \int_0^{\tau_0^{X^0}} \mathrm{e}^{-qs} F(X^{0}_s) \dd s  + \mathrm{e}^{-q \tau_0^{X^0}} \tilde{I}_F(0) \right] \\
   &=& \e_x \left[ \int_0^{\tau_0^{X^0}} \mathrm{e}^{-qs} F(\underline{X}^{0}_s) \dd s  + \mathrm{e}^{-q \tau_0^{X^0}} \tilde{I}_F(0) \right]\\
   &=& I_F(x) - \e_x \left[ \mathrm{e}^{-q \tau_0^{X^0}} I_F(0) \right]+ \e_x \left[ \mathrm{e}^{-q \tau_0^{X^0}} \tilde{I}_F(0) \right] \\
   &=& I_F(x) + (\tilde{I}_F(0)-I_F(0)) \varphi(x) .
\end{eqnarray*}
This last expression is $I_F^\lambda(x)$ with $\lambda = \tilde{I}_F(0) - I_F(0).$
\end{remark}

\section{Solution of the problem without injections}\label{sec:Vd}

First, we define the performance function of an arbitrary mean-reverting strategy at level $b$, when injections are not allowed: for $x \geq 0$, define
\[
J_d(x; b) = \e_x \left[ \int_0^{T^b} \mathrm{e}^{-qs}  F(X^{b}_s) \mathbf{1}_{\{X^{b}_s \geq b\}}  \dd s \right] .
\]
Of course, for any fixed $b \geq 0$, we have $J_d(x;b) \leq V_d(x)$ for all $x \geq 0$, which of course yields $\sup_{b \geq 0} J_d(x;b) \leq V_d(x)$ for all $x \geq 0$.

In this section, we briefly summarize the main findings of Locas \& Renaud \cite{LR2023} who showed that there exists a level $b_d$ such that $J_d(x;b_d) = V_d(x)$ for all $x \geq 0$. They also gave a characterization of this threshold. In the next theorem, we give a slightly different but equivalent statement of their main result.

\begin{theorem}\label{prop: Jd_continuous}
If $b=0$, then $J_d(x;b) = -I_F(0) \Hq(x) + I_F(x)$. If $b > 0$, then
\[
J_d(x;b) =
\begin{cases}
 D_1(b) \Wq(x) , & \text{if $0 \leq x \leq b$,}\\
 D_2(b)  \Hq(x) + I_F(x) , & \text{if $x \geq b$,}
\end{cases}
\]
where
\begin{align*}
D_1(b) &= \frac{\Hq(b) I_F^\prime (b) - \Hqprime(b) I_F(b)}{\Hq(b) \Wqprime(b) - \Hqprime(b) \Wq(b)} , & D_2(b) = \frac{\Wq(b) I_F^\prime (b) - \Wqprime(b) I_F(b)}{\Hq(b) \Wqprime(b) - \Hqprime(b) \Wq(b)} .
\end{align*}
\end{theorem}

First, it is interesting to note that $D_2(0)=-I_F(0)$. Second, it is even more interesting to note that, in the above theorem, the function $I_F$ can be replaced by $I_F^\lambda$ for any $\lambda \in \reals$. In particular, in \cite{LR2023}, the result is stated with $\tilde{I}_F : \reals \to \reals$ as defined in~\eqref{eq:Itilde}.

Finally,  it is straightforward to show that, for $b>0$, we have $J_d'(b-;b) = J_d'(b+;b)$, which means that $x \mapsto J_d(x;b)$ is a continuously differentiable function. Moreover, the optimal threshold $b_d$ is the value (of $b$) that makes $J_d(\cdot;b_d)$ a twice continuously differentiable function or, equivalently, such that $J_d'(b_d;b_d)=1$.

\begin{proposition}\label{prop.bd}
The optimal barrier level $b_d$ is characterized as follows:
\begin{enumerate}
  \item if $I_F'(0) - I_F(0) \varphi'(0) \leq 1$, then $b_d = 0$;
  \item if $I_F'(0) - I_F(0) \varphi'(0) > 1$, then $b_d>0$ and it is a solution of
\begin{equation}\label{eq.bd}
I_F'(b_d) \varphi(b_d) - I_F(b_d) \varphi'(b_d) = \varphi(b_d) - \frac{\Wq(b_d)}{\Wqprime(b_d)} \varphi'(b_d) .
\end{equation}
\end{enumerate}
In both cases, $V_d(\cdot) = J_d(\cdot;b_d) \in C^2([0,\infty))$.
\end{proposition}

Putting all of the above together gives Theorem~\ref{thm:Vd} and more. In particular, as we have an \textit{explicit} expression for $V_d$, we can use it with the optimization dichotomy (in Theorem~\ref{th.dichotomy}). Indeed, if $I_F'(0) - I_F(0) \varphi'(0) \leq 1$, then $b_d = 0$ and hence $V'_d(0+) = I_F'(0) - I_F(0) \varphi'(0) < \beta$. From Theorem~\ref{th.dichotomy}, we can deduce that, if $I_F'(0) - I_F(0) \varphi'(0) \leq 1$, then $V=V_d$. If $I_F'(0) - I_F(0) \varphi'(0) > 1$, then $b_d > 0$ and hence $V'_d(0+) = \frac{\Wq'(0+)}{\Wq'(b_d)}$. As $\Wq'(0+)=2/\sigma^2$, we can further write
\[
V'_d(0+) = \frac{(2/\sigma^2) \sqrt{\mu^2 + 2q \sigma^2}}{\alpha_1 \mathrm{e}^{\alpha_1 b_d} - \alpha_2 \mathrm{e}^{\alpha_2 b_d}} ,
\]
which can be compared numerically with $\beta$ in order to apply the dichotomy result of Theorem~\ref{th.dichotomy}.

\section{Solution of the problem with forced injections}\label{sec:Vc}

In this section, we provide a solution to the maximization problem with forced capital injections, also called the bail-out optimal dividend problem. In particular, we prove Theorem~\ref{thm:Vc}.

However, we start by showing in the next section that \textit{optimal} capital injections should only take the form of bail-out payments. This will lead to a new representation for $V_c$ but also for $V$, the value function of our main optimization problem; see Proposition~\ref{prop:Vc_reflected} below.

\subsection{Bail-out payments}

Let us  introduce the following sets of admissible controls in which capital injections can only be performed when the corresponding controlled process reaches level zero. To this purpose, we define
\[
\Pi^0_c = \{ \pi = (\ell,G) \in \Pi_c \colon G_t = \underline{G}^{\ell}_{t} \mbox{ for all } t \geq 0\}
\]
and
\[
\Pi^0 = \{\pi = (\ell, G) \in \Pi \colon \mbox{ there exists } \tau \in \mathcal T \mbox{ such that }  G_t = \underline{G}^{\ell}_{t \wedge \tau} \mbox{ for all } t \geq 0 \} ,
\]
where $\mathcal T$ is the set of all $\mathbb{F}$-stopping times.

\begin{remark}\label{remarkpi0}
In those sub-families of controls, capital injections are performed only at level zero. However, a control $\pi$ in $\Pi^0$  may differ from a process in $\Pi^0_c$ since the latter keeps the process $X^\pi$ nonnegative (by reflecting it at zero) but not necessarily the former. This means that, if $\pi \in \Pi^0$, then we might have $T^\pi < \infty$.
\end{remark}

It is possible to pair each control $\pi =(\ell,G) \in \Pi$ with a control $\pi^0 =(\ell, G^0) \in \Pi^0$ by using $G^0_t := \underline{G}^{\ell}_{t \wedge \tau}$ with $\tau = T^\pi$.

\begin{lemma}\label{lemma.2}
Let $\pi =(\ell,G) \in \Pi$ (resp.\ $\Pi_c$). Then there exists $\pi^0 =(\ell, G^0) \in \Pi^0$ (resp.\ $\Pi_c^0$) such that $J(x;\pi) \leq J(x;\pi^0)$ for all $x\geq 0 .$
\end{lemma}
\begin{proof}
Let $\pi =(\ell,G) \in \Pi$ (resp.\ $\Pi_c$) and then let $\pi^0 = (\ell, G^0)\in \Pi^0$ (resp.\ $\Pi_c^0$) where $G^0_t := \underline{G}^{\ell}_{t \wedge \tau}$ with $\tau = T^\pi$. Then, $T^\pi = T^{\pi^0}$ by Lemma~\ref{lemma.1} of Appendix~\ref{AppendixA}. In the case $\pi \in \Pi_c$, since we have $T^\pi = \infty$, then $G^0_t =  \underline{G}^{\ell}_t$ for all $t\geq 0$. In either case, by Lemma~\ref{lemma.1} of Appendix~\ref{AppendixA}, we have
\[
\int_0^{T^{\pi^0}} \mathrm{e}^{-q t} \dd G^0_t \leq \int_0^{T^{\pi}} \mathrm{e}^{-q t} \dd G_t .
\]
As a result, for $x \geq 0$,
\[
J(x;\pi) = \e_x \left[ \int_0^{T^\pi} \mathrm{e}^{-q t}(\ell_t \dd t - \beta \dd G_t) \right] \leq \e_x \left[ \int_0^{T^{\pi^0}} \mathrm{e}^{-q t}(\ell_t \dd t - \beta \dd G^0_t) \right] = J(x; \pi^0) .
\]
\end{proof}

As a consequence of the previous lemma, we deduce the following characterization of the value functions $V_c$ and $V$:
\begin{proposition}\label{prop:Vc_reflected}
For all $x\geq 0$,
$$V_c(x) =  \sup_{\pi = (\ell, G) \in \Pi_c^0} \e_x \left[ \int_0^\infty \mathrm{e}^{-q t} (\ell_t \dd t - \beta  \dd G_t) \right]$$ and
$$V(x) =   \sup_{\pi = (\ell, G) \in \Pi^0} \e_x \left[\int_0^{T^\pi} \mathrm{e}^{-q t} (\ell_t \dd t - \beta \dd G_t) \right].$$
\end{proposition}

This last representation of $V$ can be interpreted as follows: given an admissible strategy $\pi = (\ell, G) \in \Pi^0$ and its associated stopping time $\tau$,  the only injection strategies that may be optimal consist in injecting only when $X^\pi$ hits zero, by means of a reflection using $\underline{G}^{\ell}$, up to time $\tau,$ and avoiding all capital injections thereafter. In fact, in the proof of Lemma~\ref{lemma.2}, we see that the only relevant stopping times $\tau$ are the ruin time of $X^\pi$ for some given general policy $\pi \in \Pi.$  However, the Markovian structure of the problem also suggests that an optimal strategy (and its associated ruin time) is a function of the controlled process $X$. In particular, we expect the ruin time of the optimal strategy to be the hitting time of $0$, or no ruin at all. In other words, if it is optimal to stop reflecting when $X^\pi$ hits zero after some time, it should in fact be optimal to do so the first time it does (hence no reflection at all). Otherwise, it is never optimal to let the process reach ruin, and reflection is used throughout. To prove this dichotomy, we use a comparison principle for solution of the associated ODE and compare the initial values of $V_d$ and $V_c$ at. See Section~\ref{sec:dichotomy} below.

\subsection{Viscosity characterization of the value function}\label{sec:visc_sol}

 The HJB equation associated to $V_c$ is given by: for $x > 0$,
  \begin{eqnarray}\label{eq.visco} q u(x) - \sup_{0 \leq l \leq F(x)} \mathcal L_l [u](x) = 0,\end{eqnarray} with initial condition
  \begin{eqnarray}\label{eq.visco2Vc}   u'(0)  = \beta , \end{eqnarray}
where,  for $u \in C^2(0,\infty)$ and a constant $l \geq 0$,
$$
\mathcal L_l [u](x) := (\mu - l) u'(x) + \tfrac12 \sigma^2 u''(x) + l .
$$

In this section, we show that $V_c$ is a viscosity solution of~\eqref{eq.visco}-\eqref{eq.visco2Vc}, and use this to characterize the optimal threshold $b_c$ of Theorem \ref{thm:Vc}. We refer the reader to Appendix~\ref{appendix-viscosity} for a definition of viscosity solutions and a comparison principle for~\eqref{eq.visco}, as these notions will be used frequently in the rest of the paper.

\begin{remark}
Note that $V_d$ is a (classical) solution of the HJB equation~\eqref{eq.visco},  with  initial condition $u(0)=0$. See \cite{LR2023} for details.
\end{remark}

\begin{remark}\label{remarkVprime}
Since the value function  $V_c$ is nondecreasing (which simply follows from its definition) and since it is always possible to inject an amount $\epsilon$, we know that \[V_c(x+\epsilon) \geq V_c(x) \geq V_c(x+\epsilon) - \beta \epsilon, \quad x \geq 0 ,\] which can be re-written as \[0 \leq \frac{V_c(x+\epsilon) - V_c(x)}{\epsilon} \leq \beta.\] In particular, we deduce that $V_c$ is continuous.
\end{remark}

In the rest of the paper, we will need the following dynamic programming principle (DPP) for the value functions $V_c$ and $V$, which uses the characterizations obtained in Proposition~\ref{prop:Vc_reflected}. See \cite{GP2005} for a similar DPP and its proof.
\begin{proposition}\label{DPP-for-Vc}
For any stopping time $\tau$, we have
\begin{equation*}
V_c(x) = \sup_{\pi = (\ell,G) \in \Pi_c^0} \e_x \left[ \int_0^{\tau} \mathrm{e}^{-q t}(\ell_t \dd t - \beta \dd G_t) + \mathrm{e}^{-q \tau} V_c(X^{\pi}_\tau) \right] .
\end{equation*}
The result is also true for $V$ with $\Pi_c^0$ replaced by $\Pi^0$.
\end{proposition}

The following proposition presents other basic properties of this value function.

\begin{proposition} \label{prop:propertiesVc}
The value function $V_c$ is  concave. Furthermore,  there exist two constants $A<1$ and $C>0$ such that,  for $x\geq 0$,
\begin{equation}\label{ineq.boundsVc}
V_c(x) \leq C + A x .
\end{equation}
In particular, $V_c \in \mathcal W$.
\end{proposition}
\begin{proof}
To show concavity, fix $x_1,x_2 \in (0,\infty)$ and $0 < \lambda < 1$, and then define $x = \lambda x_1 + (1-\lambda) x_2$. For an arbitrary strategy $\pi^i = (\ell^i,G^i) \in \Pi_c$, set $X^{\pi^i}_{0-} = x_i$,  for each $i = 1,2$. Finally, let $X^\pi$ be the process controlled by $\pi = (\ell,G) := (\lambda \ell^1 + (1-\lambda) \ell^2, \lambda G^1 + (1-\lambda) G^2)$ and such that $X^\pi_{0-}=x$. By the linearity of the dynamics, we know that $X^\pi = \lambda X^{\pi^1} + (1-\lambda) X^{\pi^2}$ and hence $X^\pi_t \geq 0$, for all $t\geq 0$. Furthermore, since $F$ is concave, we have that, for any $t\geq 0$,
\[
\ell_t = \lambda \ell^1_t + (1-\lambda) \ell^2_t \leq \lambda F(X^{\pi^1}_t) + (1-\lambda) F(X^{\pi^2}_t) \leq F(X^\pi_t) .
\]
Therefore, $\pi \in \Pi_c$ and
\begin{eqnarray*}
  V_c(x) & \geq &  J(x;\pi) = \e_x \left[ \int_0^\infty \mathrm{e}^{-q t} (\ell_t \dd t - \beta \dd G_t) \right] \\
   &= & \lambda \e_{x_1} \left[ \int_0^\infty \mathrm{e}^{-q t} (\ell_t^1 \dd t - \beta \dd G^1_t) \right] + (1-\lambda) \e_{x_2} \left[ \int_0^\infty \mathrm{e}^{-q t}  (\ell_t^2 \dd t - \beta dG^2_t)  \right].
\end{eqnarray*}
Since $\pi^1$ and $\pi^2$ are arbitrary in $\Pi_c$, we can further write $V_c(x) \geq \lambda V_c(x_1) + (1-\lambda) V_c(x_2)$. This proves that $V_c$ is concave.

Set $A = F'(0)/(q+F'(0))$. For a constant $C>0$ to be determined, define $\bar{V}(x) = C + A x$. First, note that, for $x>0$,
\begin{eqnarray*}
q \bar{V}(x) - \sup_{0 \leq l \leq F(x)} \mathcal L_l [\bar{V}](x) & = &  q \bar{V}(x) - (\mu-F(x)) A - F(x) \\ & = & q C  - \mu A + q A x + F(x) (A - 1)\\
& \geq & q C - \mu A  - \tfrac{q}{q+F'(0)} F(0) ,
\end{eqnarray*}
where we used the fact that $F(x) \leq F(0) + A (q+F'(0)) x$, due to its concavity. Clearly, for $C \geq V_c(0)$ sufficiently large, we have
\begin{equation}\label{eq:supersolution}
q \bar{V}(x) - \sup_{0 \leq l \leq F(x)} \mathcal L_l [\bar{V}](x) \geq 0 , \quad x>0 .
\end{equation}
Second, 
choose an arbitrary $\pi=(\ell, \underline{G}^{\ell}) \in \Pi_c^0$ and set $\tau_0 = \inf \left\lbrace t > 0 \colon X^\pi_t = 0 \right\rbrace$. Using It\^o's Formula and the fact that $\underline{G}^{\ell}_t=0$ for all $t \in [0,\tau_0]$, we can write
\begin{eqnarray*}
\bar{V}(x) &=& \e_x \left[ \int_0^{\tau_0} \mathrm{e}^{-qt} \left( q \bar{V}(X^\pi_t) - (\mu-\ell_t) \bar{V}'(X^\pi_t) \right) \mathrm{d}t + \mathrm{e}^{-q \tau_0} \bar{V}(X^\pi_{\tau_0}) \right] \\
& \geq & \e_x \left[ \int_0^{\tau_0} \mathrm{e}^{-qt} \left( q \bar{V}(X^\pi_t) - \sup_{0 \leq l \leq F(X^\pi_t)} \mathcal L_l [\bar{V}](X^\pi_t) + \ell_t \right) \mathrm{d}t + \mathrm{e}^{-q \tau_0} C \right] \\
& \geq & \e_x \left[ \int_0^{\tau_0} \mathrm{e}^{-qt} \ell_t \mathrm{d}t + \mathrm{e}^{-q \tau_0} V_c(0) \right] ,
\end{eqnarray*}
where, in the  last inequality, we used~\eqref{eq:supersolution} and the fact that $C \geq V_c(0)$. By taking a supremum over $\pi \in \Pi_c^0$ and using the DPP for $V_c$, we find that $V_c \leq \bar V.$ 
\end{proof}

\begin{proposition}\label{prop:Vc-derivative-at-zero}
The value function $V_c$ is differentiable at $x=0$ with derivative $V_c^\prime(0+) = \beta$.
\end{proposition}

\begin{proof}
Take $x = 0$ and $\epsilon < \mu/q$. Choose an arbitrary $\pi = (\ell,\underline{G}^\ell) \in \Pi_c^0$ and set $\tau_\epsilon = \inf\{t \geq 0 : X^{\pi}_t = \epsilon\}$. Since $\underline{G}^\ell$ is continuous, $X^{\pi}$ is also continuous and $X^\pi_{\tau_\epsilon} = \epsilon$.

Consider the strategy $\tilde \pi = (\tilde \ell,\underline{G}^{\tilde \ell}) \in \Pi^0_c,$ with $\tilde \ell \equiv F(\epsilon).$ Since $\ell_t  \leq F(\epsilon)$ on $[0, \tau_\epsilon]$, Lemma \ref{lemma.1} of Appendix~\ref{AppendixA} gives us that $X^{\tilde \pi}_t \leq X^{\pi}_t$, $\p_0$-a.s. Therefore, $\tau_\epsilon \leq \tilde \tau_\epsilon := \inf\{t\geq 0 : X^{\tilde \pi}_t = \epsilon\}$, $\p_0$-a.s.

Then, by Proposition~\ref{prop:Vc_reflected} and the DPP for $V_c$,
\begin{eqnarray*}
V_c(0) &=&  \sup_{\pi = (\ell, \uGell) \in \Pi^0_c} \e_0 \left[ \int_0^{\tau_\epsilon} \mathrm{e}^{-q t}(\ell_t \dd t - \beta \dd \uGell_t) + \mathrm{e}^{-q \tau_\epsilon} V_c(\epsilon) \right] \\
& \leq & \sup_{\pi = (\ell, \uGell) \in  \Pi^0_c} \e_0 \left[  \beta \int_0^{\tau_\epsilon} \mathrm{e}^{-q t}(\ell_t \dd t - \dd \uGell_t) + \mathrm{e}^{-q \tau_\epsilon} V_c(\epsilon) \right] \\
&  = & \sup_{\pi = (\ell, \uGell) \in  \Pi^0_c} \e_0 \left[  \beta \int_0^{\tau_\epsilon} \mathrm{e}^{-q t}(  \mu \dd t - \dd X^{\pi}_t) + \mathrm{e}^{-q \tau_\epsilon} V_c(\epsilon) \right]  \\
&  = & \sup_{\pi = (\ell, \uGell) \in  \Pi^0_c} \e_0 \left[ - \beta \epsilon \mathrm{e}^{- q \tau_\epsilon}  +  \beta \int_0^{\tau_\epsilon} \mathrm{e}^{-q t}(\mu - q X^{\pi}_t) \dd t + \mathrm{e}^{-q \tau_\epsilon} V_c(\epsilon) \right] \\
& \leq & \sup_{\pi = (\ell, \uGell) \in  \Pi^0_c} \e_0 \left[ - \beta X^{\tilde \pi}_{\tilde \tau_\epsilon} \mathrm{e}^{- q \tau_\epsilon}  +  \beta \int_0^{\tau_\epsilon} \mathrm{e}^{-q t}(\mu - q X^{\tilde \pi}_t) \dd t + \mathrm{e}^{-q \tau_\epsilon} V_c(\epsilon) \right] \\
& \leq & \sup_{\pi = (\ell, \uGell) \in  \Pi^0_c} \e_0 \left[ - \beta X^{\tilde \pi}_{\tilde \tau_\epsilon} \mathrm{e}^{- q \tau_\epsilon}  +  \beta \int_0^{\tilde \tau_\epsilon} \mathrm{e}^{-q t}(\mu - q X^{\tilde \pi}_t) \dd t + \mathrm{e}^{-q \tau_\epsilon} V_c(\epsilon) \right] \\
& = & \e_0 \left[ \beta \int_0^{\tilde \tau_\epsilon} \mathrm{e}^{-q t}(F(\epsilon) \dd t - \dd \underline{G}^{\tilde \ell}_t) \right]  +  \sup_{\pi = (\ell, \uGell) \in \Pi^0_c} \e_0 \left[ \mathrm{e}^{-q \tau_\epsilon} V_c(\epsilon)\right].
\end{eqnarray*}
where, in the first inequality, we used that $\beta >1$. Also, we have used that $\e_x\left[ \int_0^{\tau_\epsilon} \mathrm{e}^{-qt} \dd B_t\right]=0$ (which follows from the fact that $\tau_\epsilon \leq \tilde \tau_\epsilon$ and $\e_x\left[ \int_0^{\tilde \tau_\epsilon} \mathrm{e}^{-2 qt} \dd t\right]<\infty$) and the fact that $\mu - q X^{\tilde \pi}_t \geq \mu - q \epsilon > 0$ for $t \leq \tilde \tau_\epsilon.$

Note that $$0 \leq \mathrm{e}^{-q \tau_\epsilon} V_c(\epsilon) \leq V_c(\epsilon) - (1 - \mathrm{e}^{-q \tilde \tau_\epsilon} ) \min(V_c(\epsilon),0), \mbox{ $\p_0$-a.s.}$$
 Consequently,
 \begin{eqnarray*}
  V_c(\epsilon) - V_c(0) &\geq &  \beta \e_0 \left[ \int_0^{\tilde \tau_\epsilon} \mathrm{e}^{-q t}( \dd \underline{G}^{\tilde \ell}_t- F(\epsilon) \dd t )\right] + \min(V_c(\epsilon),0) \left( 1 - \e_0  \left[\mathrm{e}^{-q \tilde \tau_\epsilon} \right]  \right)\\
   &\geq &  \beta \e_0 \left[ \int_0^{\tilde \tau_\epsilon} \mathrm{e}^{-q t} \dd \underline{G}^{\tilde \ell}_t\right] +(\min(V_c(\epsilon),0) - \beta F(\epsilon)/q) \left( 1- \e_0 \left[\mathrm{e}^{-q \tilde \tau_\epsilon}  \right] \right).
\end{eqnarray*}

Using notation and fluctuation identities from Section~\ref{sec:notation}, we can show that
\[
\lim_{\epsilon \downarrow 0} \epsilon^{-1} \e_0 \left[ \int_0^{\tilde \tau_\epsilon}\mathrm{e}^{-qt} \dd \underline{G}^{\tilde \ell}_t \right] = \lim_{\epsilon \downarrow 0} \epsilon^{-1} u_{0,\epsilon}(0) = \lim_{\epsilon \downarrow 0} \epsilon^{-1} \left[ \frac{\Zq(0)}{\Zq(\epsilon)} \left(\overline{\Zq}(\epsilon)+\frac{\mu}{q} \right) - \left(\overline{\Zq}(0) + \frac{\mu}{q} \right) \right] = 1
\]
and
\[
\lim_{\epsilon \downarrow 0} \epsilon^{-1} \left( \min(V_c(\epsilon),0) - \beta \frac{F(\epsilon)}{q} \right) \left(1- \e_0[\mathrm{e}^{-q\tilde \tau_\epsilon}] \right) = \min \left( V_c(0) - \beta \frac{F(0)}{q},0 \right) \lim_{\epsilon \downarrow 0} \frac{\Zq(\epsilon)-\Zq(0)}{\Zq(\epsilon) \epsilon} = 0 ,
\]
where with an abuse of notation $u_{0,\epsilon}$ and $\Zq(x)$, as defined in Section~\ref{sec:notation}, have been used with $\mu$ replaced by $\mu - F(\epsilon)$. As a result,
\begin{eqnarray*}
  \liminf_{\epsilon \downarrow 0} \frac{V_c(\epsilon)-V_c(0)}{\epsilon}
   &  \geq & \beta. \\
\end{eqnarray*}
By Remark~\ref{remarkVprime}, $\displaystyle{\limsup_{\epsilon \downarrow 0} \frac{V_c(\epsilon)-V_c(0)}{\epsilon} \leq \beta}$ and the result follows.
\end{proof}

\begin{theorem}[Viscosity Characterization]\label{thm:viscotity_characterization_Vc}
The value function $V_c$ is the unique viscosity solution in $\mathcal W$ of the HJB equation in~\eqref{eq.visco} with initial condition in~\eqref{eq.visco2Vc}.
\end{theorem}

We already know, from Proposition~\ref{prop:Vc-derivative-at-zero}, that $V_c$ satisfies the initial condition in~\eqref{eq.visco2Vc}. The proof that it is a viscosity solution of~\eqref{eq.visco} is standard; see, for instance, \cite{CLVR2020}, \cite{DV2007}, \cite{LVPV2008}. The uniqueness in $\mathcal W$ follows from the comparison principle in Theorem~\ref{thm:comparison} of Appendix~\ref{appendix-viscosity}.

A first consequence of the viscosity characterization is that $V_c$ is a smooth function.
\begin{proposition}\label{prop.C1}
The value function $V_c$ is continuously differentiable on $[0,\infty)$.
\end{proposition}
A proof of this last proposition is provided in Appendix~\ref{preuve-prop.C1}.

The concavity and smoothness of $V_c$ allow us to expect that the following (analytical) values will correspond to (probabilistic) levels for optimal capital injections and optimal dividend payments: define
$$
a_c = \inf\{x \geq 0 : V_c'(x) < \beta\} \quad \mbox{and} \quad b_c = \inf\{x\geq 0 : V_c'(x) = 1\}.
$$
By concavity and since $V_c'(0+) = \beta$, we know that $V_c'(x) = \beta$ for all $x \leq a_c$. In particular, if $a_c > 0$, we deduce that $V_c(x) = V_c(a_c) - \beta (a_c-x),$ for  $x <  a_c.$ In words, this means we expect that it should be optimal to inject an amount $a_c-x$ when the process is at level $x \in (0,a_c)$, i.e., the process jumps upward from $x$ to $a_c$ by paying $\beta(a_c-x)$. Intuitively, the condition $V'(x) = \beta$ tells us when to inject. On the other hand, the threshold $b_c$ should correspond to the barrier of an optimal bang-bang dividend strategy. Indeed, concavity of $V_c$ implies that $V_c' \geq 1$ on $[0,b_c]$ and $V_c' \leq 1$ on $[b_c,\infty).$ The significance of this observation comes from the HJB equation~\eqref{eq.visco}, in which the supremum is attained for $l=0$ when $V_c'(x) > 1$, and for $l = F(x)$ when $V_c'(x) < 1$.

The next proposition is more specific about those values. In particular, it is (intuitively) coherent with Proposition~\ref{prop:Vc_reflected}.

\begin{proposition}\label{prop:b_c_exists}
We have  $a_c =0$ and $0 < b_c < \infty$.
\end{proposition}
\begin{proof}
Assume $a_c > 0$.  Since $V_c$ is concave and such that $V'_c(0+)=\beta$, we must have $V_c(x) = V_c(0) + \beta x$ for all $x\leq a_c$. As a consequence, $V_c$ is twice continuously differentiable on $(0,a_c)$ and, since it is a viscosity solution of the HJB equation in~\eqref{eq.visco}, it is also a classical solution on $(0,a_c)$. Therefore, plugging this expression in the HJB equation, we find that
$$
q V_c(x) - \sup_{0\leq l \leq F(x)} \mathcal L_l [V_c](x) = q V_c(0) + q \beta x - \mu \beta .
$$
Clearly, this last expression cannot be equal to zero for all $x \in (0,a_c)$, yielding a contradiction of the assumption that $a_c>0$.

Again, since $V_c$ is concave and such that $V'_c(0+)=\beta$, we have $b_c > 0$. Now, assume $b_c = \infty$, i.e., $V_c'(x) > 1$ for all $x \geq 0.$ From Proposition~\ref{prop:propertiesVc}, we know that $\frac{V_c(x)}{C + Ax} \leq 1$ for all $x \geq 0$, where $C>0$ and $A<1$. Since $V_c$ is concave, we also have that $V_c(x) \geq V_c(0) + V^\prime_c(x) x$ for all $x \geq 0$. Consequently,
\begin{equation*}
1 \geq \limsup_{x\to \infty} \frac{V_c(x)}{C + Ax} \geq \limsup_{x\to \infty} \frac{V_c(0) + V^\prime_c(x)x}{C + Ax} \geq \frac{1}{A} > 1 ,
\end{equation*}
which is a contradiction.
\end{proof}

The fact that $b_c$ lies in $(0,\infty)$ splits the HJB equation, of which $V_c$ is a viscosity solution, into the following two ODEs: on $[0,b_c]$,
\begin{eqnarray}\label{eq.viscoVc1a}
q u(x) - \mu u'(x) - \tfrac12 \sigma^2 u''(x) & = & 0 , \quad 0 < x < b_c, \\
u'(0+) & = & \beta , \label{eq.viscoVc1b}\\
u'(b_c) & = & 1, \label{eq.viscoVc1c}
\end{eqnarray}
and, on $[b_c,\infty)$,
\begin{eqnarray}\label{eq.viscoVc2a}
q u(x) - (\mu - F(x)) u'(x) - \tfrac12 \sigma^2 u''(x) - F(x) & = & 0, \quad b_c < x < \infty, \\
u'(b_c) & = & 1.\label{eq.viscoVc2b}
\end{eqnarray}

These two ODEs further characterize the value function $V_c$ as a classical solution, i.e., as a twice continuously differentiable solution, of each of these ODEs:
\begin{theorem}\label{thm:V_c_is_C2}
The restriction of $V_c$ to $[0,b_c]$ is the unique classical solution in $\mathcal W$ of~\eqref{eq.viscoVc1a}-\eqref{eq.viscoVc1b}-\eqref{eq.viscoVc1c}, while the restriction of $V_c$ to $[b_c,\infty)$ is the unique classical solution in $\mathcal W$ of~\eqref{eq.viscoVc2a}-\eqref{eq.viscoVc2b}.
Furthermore, we have $\lim_{x \uparrow b_c } V_c''(x) = \lim_{x \downarrow b_c } V_c''(x)$ and hence $V_c$ is twice continuously differentiable on $[0,\infty)$.
\end{theorem}

\begin{proof}
First, we show that each restriction is a viscosity solution of the corresponding ODE, i.e., of~\eqref{eq.viscoVc1a} and~\eqref{eq.viscoVc2a}, respectively.

Let $x\in (0,b_c).$ Following the definition of viscosity subsolutions (see Definition \ref{def.visco} of Appendix \ref{appendix-viscosity}), let $\phi \in C^2(0,b_c)$ be such that $\phi(x) = V_c(x)$ and $\phi\geq V_c$. We need to show that $\phi$ is a subsolution of~\eqref{eq.viscoVc1a}. By definition of $b_c$ and the fact that $a_c = 0$, we have $1 < V_c'(x) = \phi'(x) < \beta$. Since $V_c$ is a viscosity subsolution of~\eqref{eq.visco} and
$$
0 \leq \sup_{l \leq F(x)} \mathcal L_l [\phi](x) = \mu \phi'(x) + \tfrac12 \sigma^2 \phi''(x),
$$
then $\phi$ satisfies~\eqref{eq.viscoVc1a} at the point $x$ with an inequality ($\leq$) instead of an equality.  This shows that $V_c$ is a viscosity subsolution of~\eqref{eq.viscoVc1a}.

Let $x\in (b_c,\infty)$ and let $\phi \in C^2(b_c,\infty)$ be such that $\phi(x) = V_c(x)$ and $\phi\geq V_c$. As before, we need to show that $\phi$ is a subsolution of~\eqref{eq.viscoVc2a}. By definition of $b_c$, we have $V_c'(x) = \phi'(x) \leq 1$. Since $V_c$ is a viscosity subsolution of~\eqref{eq.visco} and
$$
0 \leq \sup_{l \leq F(x)} \mathcal L_l [\phi](x) = (\mu-F(x)) \phi'(x) + \tfrac12 \sigma^2 \phi''(x) + F(x),
$$
then $\phi$ satisfies~\eqref{eq.viscoVc2a} at the point $x$ with an inequality ($\leq$) instead of an equality. This shows that $V_c$ is viscosity subsolution of~\eqref{eq.viscoVc2a}.

The supersolution property is proved similarly. It is also clear that $V_c$ satisfies~\eqref{eq.viscoVc1b}-\eqref{eq.viscoVc1c}-\eqref{eq.viscoVc2b}.

Standard results from ordinary differential equations theory imply that the solution of~\eqref{eq.viscoVc1a}-\eqref{eq.viscoVc1b}-\eqref{eq.viscoVc1c} (resp.\ \eqref{eq.viscoVc2a}-\eqref{eq.viscoVc2b}) is in fact unique and in $C^2$. Since any classical solution is also a viscosity solution, we know that $V_c$ is $C^2$ on $[0,b_c)$ and $(b_c, \infty)$ by uniqueness. In particular, the left and right second derivatives of $V_c$ at the point $b_c$ are well defined. Since $V_c'(b_c) = 1$, we readily find from~\eqref{eq.viscoVc1a} and~\eqref{eq.viscoVc2a} that the left and right second derivatives are equal at $x = b_c$. Therefore, $V_c \in C^2[0,\infty).$
\end{proof}

\subsection{Mean-reverting strategies with forced injections}\label{sec:Fluctuation_theory}

From the previous section, it seems now natural to study the family of reflected mean-reverting strategies, i.e., bang-bang strategies $(\ell^b,\underline{G}^{b})$ consisting in paying out dividends at maximal rate above level $b>0$ and (always) injecting at level $0$. We will compute the performance function of an arbitrary reflected mean-reverting strategy. As a by-product, an equation for a candidate optimal barrier level can be deduced. Then, in the next section, we will show that this optimal barrier level is well defined and indeed equal to $b_c$, and that the value function of the control problem is equal to the performance function of the corresponding reflected mean-reverting strategy. More details in Section~\ref{sect:optimal-reflected-strategy}.

Now, let us specify the performance function of an arbitrary reflected mean-reverting strategy at level $b$: for $x \geq 0$, define
\[
J_c(x; b) = \e_x \left[ \int_0^\infty \mathrm{e}^{-qs} \left( F(\underline{X}^{b}_s) \mathbf{1}_{\{\underline{X}^{b}_s \geq b\}}  \dd s - \beta \dd \underline{G}^b_s \right) \right] .
\]
Of course, for any fixed $b \geq 0$, we have $J_c(x;b) \leq V_c(x)$ for all $x \geq 0$, which of course yields $\sup_{b \geq 0} J_c(x;b) \leq V_c(x)$ for all $x \geq 0$.

\begin{theorem}\label{prop: Jc_continuous}
For a fixed $b > 0$, we have
\[
J_c(x;b) =
\begin{cases}
\left( C_1(b) - \beta \tfrac{\sigma^2}{2} C_3(b) \right) \Zq(x) + \beta  \tfrac{\sigma^2}{2} \Wq(x) , & \text{if $0 \leq x < b$,}\\
\left(C_2(b) - \beta \tfrac{\sigma^2}{2} C_4(b)  \right) \Hq(x) + I_F(x) , & \text{if $x \geq b$,}
\end{cases}
\]
where
\begin{align*}
C_1(b) &= \frac{\Hq(b) I_F^\prime (b) - \Hqprime(b) I_F(b)}{\Hq(b) \Zqprime(b) - \Hqprime(b) \Zq(b)} , & C_2(b) = \frac{\Zq(b) I_F^\prime (b) - \Zqprime(b) I_F(b)}{\Hq(b) \Zqprime(b) - \Hqprime(b) \Zq(b)} ,\\
C_3(b) &= \frac{\Wqprime (b) \Hq (b) - \Wq (b) \Hqprime (b)}{\Hq(b) \Zqprime(b) - \Hqprime(b) \Zq(b)} , & C_4(b) = \frac{\Zq (b) \Wqprime (b) - \Zqprime (b) \Wq (b)}{\Hq(b) \Zqprime(b) - \Hqprime(b) \Zq(b)} .
\end{align*}
\end{theorem}
\begin{proof}
For the purpose of this proof, we define the following two functions:
\[
f(x) =  \e_x \left[ \int_0^\infty \mathrm{e}^{-qt} F(\underline{X}^{b}_t) \mathbf{1}_{\{\underline{X}^{b}_t > b\}} \mathrm{d}t \right] \quad \text{and} \quad g(x) = \e_x \left[ \int_0^\infty \mathrm{e}^{-qt} \mathrm{d} \underline{G}^b_t \right] .
\]

First, we compute $f(x)$. For $x>0$, by the strong Markov property of $\underline{X}^{b}$ and then by the fact that the processes $\underline{X}^{b}$ and $X^b$ have the same distribution up to their first passage at level $0$ (with respect to $\p_x$), we can write
\[
f(x) = \e_x \left[ \int_0^{\tau_0^{X^b}} \mathrm{e}^{-qt} F(X^b_t) \ind_{\{X^b_t\geq b\}}  \dd t \right] + \e_x \left[ \mathrm{e}^{-q \tau_0^{X^b}} \right] f(0) .
\]
Note that Lemma~\ref{Lemma:fluctuation-identity} in the Appendix gives us an expression for $\e_x \left[ \mathrm{e}^{-q \tau_0^{X^b}} \right]$, while the following identity has been computed in Proposition 3.2 of \cite{LR2023}:
\[
\e_x \left[ \int_0^{\tau_0^{X^b}} \mathrm{e}^{-qt} F(X^b_t) \ind_{\{X^b_t\geq b\}}  \dd t \right] =
\begin{cases}
\Wq(x) \left( \frac{I_F^\prime(b) \Hq(b) - I_F(b) \Hqprime(b)}{\Wqprime(b) \Hq(b) - \Wq(b) \Hqprime(b)} \right) & \text{if $0 \leq x \leq b$,}\\
I_F(x) + \Hq(x) \left( \frac{I_F^\prime(b) \Wq(b) - I_F(b) \Wqprime (b)}{\Wqprime(b) \Hq(b) - \Wq(b) \Hqprime(b)} \right) & \text{if $x \geq b$.}
\end{cases}
\]

Similarly, by the strong Markov property of $\underline{X}^{b}$ and then by the fact that the processes $\underline{X}^{b}$ and $Y$ have the same distribution up to their first passage at level $b$ (for example, with respect to $\p_0$), we can write
\[
f(0) = \e_0 \left[ \mathrm{e}^{-q \tau^Y_b} \right] f(b) = \frac{1}{\Zq(b)} f(b) .
\]
In conclusion, after solving for $f(b)$ in the obtained expression, we get
\[
f(x) =
\begin{cases}
C_1(b) \Zq(x) , & \text{if $0 \leq x < b$,}\\
I_F(x) + C_2(b) \Hq(x) , & \text{if $x \geq b$.}
\end{cases}
\]

Second, we compute $g(x)$. For $0 \leq x \leq b$, using Markovian arguments as above we can write
\[
g(x) = \e_x \left[ \int_0^{\tau^Y_b} \mathrm{e}^{-qt}  \dd  \underline{G}_t \right] + \e_x \left[ \mathrm{e}^{-q \tau^Y_b} \right] g(b) = u_{0,b}(x) + \frac{\Zq(x)}{\Zq(b)} g(b).
\]
In particular, we can compute $g'(0+) = u_{0,b}'(0+) = -1$ since $\Zqprime(0) = 0.$
Moreover, for any $x \geq 0$, we have
\begin{equation}\label{eq.g(x)} g(x) =  \e_x \left[ \mathrm{e}^{-q \tau_0^{X^b}} \right] g(0) = h(x) g(0),\end{equation} in which $h$ is defined in  Lemma~\ref{Lemma:fluctuation-identity} of Appendix \ref{AppendixC} and given by the expression
\[h(x) = \left\{
           \begin{array}{ll}
             \Zq (x) - \Wq(x) / C_3(b), & \hbox{if $0 \leq x < b$,} \\
             \Hq(x) C_4(b)/C_3(b), & \hbox{if $x \geq b$.}
           \end{array}
         \right.
\] Since $h'(0+)$ exists, we deduce from (\ref{eq.g(x)}) that $-1 = g'(0+) = h'(0+) g(0)$ which implies that $g(0) = -\frac{1}{h'(0+)} = \frac{C_3(b)}{\Wqprime(0)} = \frac{\sigma^2 C_3(b)}{2}.$ The result follows.
\end{proof}

\begin{remark}
Given that $V_c(x) \geq J_c(x;b)$ for any $x,b\geq0$, it is now clear from~Theorem \ref{prop: Jc_continuous} that $V_c(x) > -\infty$.
\end{remark}

It is easy to verify that $J_c(\cdot;b)$ is continuous. Also, for $x \in (0,b)$, we have
\[
J_c^\prime(x;b) =
\left( C_1(b) - \beta \tfrac{\sigma^2}{2} C_3(b) \right) \Zqprime(x) - \beta \tfrac{\sigma^2}{2} \Wqprime(x)
\]
and, for $x \in (b,\infty)$, we have
\[
J_c^\prime(x;b) = C_2(b) \Hqprime(x) + I_F^\prime(x)- \beta \tfrac{\sigma^2}{2} C_4(b) \Hqprime(x).
\]
It is also easy to verify that
\[
C_1(b) \Zqprime(b) = I_F^\prime(b) + C_2(b) \Hqprime(b) \quad \mbox{ and }  \quad C_3(b) \Zqprime(x) + \Wqprime(x) = C_4(b) \Hqprime(x)
\]
from the definitions of $C_1$, $C_2,$ $C_3$ and $C_4$. The details are left to the reader.

In particular, we have the following smooth-fit condition:
\begin{lemma}\label{lem:smooth_fit}
For $b>0$, we have $J_c^\prime(b-;b) = J_c^\prime(b+;b)$ and hence $J_c(\cdot;b)$ is continuously differentiable on $[0,\infty)$.
\end{lemma}

In general, $J_c(\cdot;b)$ is not twice continuously differentiable but it is clearly twice continuously differentiable on $(0,b)$ and on $(b,\infty)$. Moreover, as $\Wqprime(0) = \tfrac{2}{\sigma^2}$ and $\Zqprime(0) = 0,$ we have  $J_c^\prime(0+;b) = \beta$. Also, note that, as a linear combination of solutions to Equations~\eqref{eq:ODE_hom} and~\eqref{eq:ODE_F}, $J_c(\cdot;b)$ is a solution to
\[
\frac{\sigma^2}{2} u''(x) + \mu u'(x) - q u(x) = 0 , \; x \in (0,b)
\]
and
\[
\frac{\sigma^2}{2} u''(x) + (\mu - F(x))u'(x) - q u(x) + F(x) = 0 , \; x \in (b,\infty) .
\]

\subsection{The optimal reflected mean-reverting strategy}\label{sect:optimal-reflected-strategy}

We are now ready to solve our control problem. Using the representation of the performance function $J_c(\cdot;b_c)$, as given in Theorem~\ref{prop: Jc_continuous}, with barrier level $b_c$ defined in Section~\ref{sec:visc_sol}, we will show it is equal to $V_c$.

\begin{theorem} \label{prop:VcJ}
The optimal barrier of the mean-reverting strategy with forced injections is $b_c$ and $J_c(x; b_c) = V_c(x)$ for all $x \geq 0$.
\end{theorem}
\begin{proof}
One one hand, recall that $\Wq$ and $\Zq$ are two linearly independent solutions to
\[
q u(x) - \mu u'(x) - \frac{\sigma^2}{2} u''(x) = 0, \quad x >0,
\]
such that $\Wqprime(0) = \tfrac{2}{\sigma^2}$ and $\Zqprime(0) = 0$. Recall, from Theorem~\ref{thm:V_c_is_C2}, that the restriction of $V_c$ to $[0,b_c]$ is also a (classical) solution of this ODE and that $V'_c(0+)=\beta$. Consequently, for some $B \in \reals$, we have
$V_c(x) = B_1 \Zq(x) + \beta \tfrac{\sigma^2}{2} \Wq(x)$ for all $x \in [0,b_c)$. Similarly, we have that $I_F \in \mathcal W$ is a solution to
\[
q u(x) - (\mu - F(x)) u^\prime(x) - \frac{\sigma^2}{2} u^{\prime \prime}(x) = F(x), \quad x >0,
\]
while $\Hq \in \mathcal W$ is a solution to the homogeneous version of this same ODE. Using again Theorem~\ref{thm:V_c_is_C2}, we have $V_c(x) = B_2 \Hq(x) + I_F(x)$ for all $x \in [b_c,\infty)$, for some $B_2 \in \reals$.

On the other hand, from Theorem~\ref{prop: Jc_continuous}, we have
\[
J_c(x;b_c) =
\begin{cases}
D_1 \Zq(x) + \beta  \tfrac{\sigma^2}{2} \Wq(x) & \text{if $0 \leq x \leq b_c$,}\\
D_2 \Hq(x) + I_F(x) & \text{if $x \geq b_c$,}
\end{cases}
\]
with $D_1 = C_1(b_c) - \beta \tfrac{\sigma^2}{2} C_3(b_c)$ and $D_2 = C_2(b_c) - \beta \tfrac{\sigma^2}{2} C_4(b_c)$, and from Lemma~\ref{lem:smooth_fit} we have that $J_c(\cdot ; b_c)$ is continuously differentiable. The continuity and smoothnes of $J_c(\cdot ; b_c)$ at $x=b_c$ means that $(D_1,D_2)$ is a solution of
\begin{align*}
D_1 \Zq(b_c) + \beta  \tfrac{\sigma^2}{2} \Wq(b_c) &= D_2 \Hq(b_c) + I_F(b_c) ,\\
D_1 \Zqprime(b_c) + \beta  \tfrac{\sigma^2}{2} \Wqprime(b_c) &= D_2 \Hqprime(b_c) + I_F'(b_c).
\end{align*}
From Proposition~\ref{prop.C1} (or Theorem~\ref{thm:V_c_is_C2}), we know that $V_c$ is also continuously differentiable, meaning that $(B_1, B_2)$ is a solution of the same system of equations. We conclude that $(D_1, D_2) = (B_1, B_2)$ and the result follows.
\end{proof}

\begin{corollary}
The optimal barrier level $b_c$ is characterized by the following supercontact condition:
\begin{equation*}\label{eq:smooth_fit_1}
J_c^{\prime \prime}(b_c-; b_c) = J_c^{\prime \prime}(b_c+; b_c)
\end{equation*}
or, equivalently, by
\begin{equation*}\label{eq:smooth_fit_2}
J_c^\prime(b_c-; b_c) = 1 = J_c^\prime(b_c+; b_c) .
\end{equation*}
In particular, $b_c$ satisfies
\begin{equation} \label{eq.bc1}
\frac{1 - \beta \tfrac{\sigma^2}{2} \Wqprime(b_c)}{\Zqprime(b_c)} = C_1(b_c) - \beta \tfrac{\sigma^2}{2} C_3(b_c)
\end{equation}
and
\begin{equation} \label{eq.bc2}
\frac{1 - I_F^\prime(b_c)}{\Hqprime(b_c)} = C_2(b_c) - \beta \tfrac{\sigma^2}{2} C_4(b_c).
\end{equation}
\end{corollary}

Theorem~\ref{thm:Vc} then follows from the above corollary and from Theorem~\ref{prop:VcJ}.

\section{Solution of the main optimization problem} \label{sec:dichotomy}

In this section, we provide a solution to the main optimization problem. Recall first that, by definition of each value function, we have
\[
V(x) \geq \max \left(V_d(x), V_c(x) \right) ,
\]
for all $x \geq 0$. In what follows, we give a proof of Theorem~\ref{th.dichotomy} by proving the reverse inequality. It is based on repeated applications of the comparison principle in Theorem~\ref{thm:comparison} of Appendix~\ref{appendix-viscosity}.

Let us recall the HJB equation given by~\eqref{eq.visco}: for $x>0$,
\begin{eqnarray*}
q u(x) - \sup_{0 \leq l \leq F(x)} \mathcal L_l [u](x) = 0 .
\end{eqnarray*}
Note that $V_d$ is, in particular, a viscosity solution in $\mathcal W$ of this HJB equation with initial condition $V_d(0) = 0$, while $V_c$ is a viscosity solution in $\mathcal W$ with initial condition $V_c^\prime(0+) = \beta$ (cf.\ Theorem~\ref{thm:viscotity_characterization_Vc}). In order to compare the three value functions $V$, $V_d$ and $V_c$ using the above comparison principle, we must have that $V$ is also a viscosity solution of this HJB equation with its own initial condition.

\begin{theorem}[Viscosity Characterization]\label{thm:viscosity-characterization}
The value function $V$ is a continuous viscosity solution in $\mathcal W$ of~\eqref{eq.visco} with initial condition
\begin{eqnarray}\label{eq.visco2}
\min\left(V(0), \beta - V'(0+) \right) = 0.
\end{eqnarray}
\end{theorem}
\begin{proof}
The proof that $V$ is a viscosity solution of~\eqref{eq.visco} is classical and omitted.  Applying the argument in Remark~\ref{remarkVprime} to $V$, we get that, for all $x\geq 0$ and $\epsilon >0$,
\begin{equation*}
0 \leq \frac{V(x+\epsilon) - V(x)}{\epsilon} \leq \beta .
\end{equation*}
Therefore, $V$ is continuous and, moreover, we have
\begin{equation}\label{ineq.boundsV}
\limsup_{\epsilon \downarrow 0} \frac{V(x+\epsilon) - V(x)}{\epsilon} \leq \beta
\end{equation}

To prove that $V \in \mathcal W$, we will proceed as in the proof of Proposition~\ref{prop:propertiesVc}. Define $\bar{V}(x) = V(0) + \frac{\mu \beta}{q} + \beta x$. First, note that, for $x>0$,
\begin{equation}\label{eq:supersolution2}
q \bar{V}(x) - \sup_{0 \leq l \leq F(x)} \mathcal L_l [\bar{V}](x) = q V(0) + q \beta x + F(x) (\beta -1) \geq 0 .
\end{equation}
Second, choose an arbitrary $\pi=(\ell, \underline{G}^{\ell}) \in \Pi^0$ and set $\tau_0 = \inf \left\lbrace t > 0 \colon X^\pi_t = 0 \right\rbrace$. Using It\^o's Formula and the fact that $\underline{G}^{\ell}_t=0$ for all $t \in [0,\tau_0]$, we can write
\begin{eqnarray*}
\bar{V}(x) &=& \e_x \left[ \int_0^{\tau_0} \mathrm{e}^{-qt} \left( q \bar{V}(X^\pi_t) - (\mu-\ell_t) \bar{V}'(X^\pi_t) \right) \mathrm{d}t + \mathrm{e}^{-q \tau_0} \bar{V}(0) \right] \\
& \geq & \e_x \left[ \int_0^{\tau_0} \mathrm{e}^{-qt} \left( q \bar{V}(X^\pi_t) - \sup_{0 \leq l \leq F(X^\pi_t)} \mathcal L_l [\bar{V}](X^\pi_t) + \ell_t \right) \mathrm{d}t + \mathrm{e}^{-q \tau_0} \bar{V}(0) \right] \\
& \geq & \e_x \left[ \int_0^{\tau_0} \mathrm{e}^{-qt} \ell_t \mathrm{d}t + \mathrm{e}^{-q \tau_0} V(0) \right],
\end{eqnarray*}
where in the last inequality we used~\eqref{eq:supersolution2} and the fact that $\bar{V}(0) \geq V(0)$. By taking a supremum over all  $\pi \in \Pi^0$, and using the DPP for $V$, we find that $V\leq \bar V.$ This proves that $V \in \mathcal W$.

All is left to prove is that $V$ satisfies the initial condition in~\eqref{eq.visco2}. Note that $V(0)\geq 0$. First, if $V(0) = 0$, then by the Comparison Principle (see Theorem~\ref{thm:comparison} in Appendix~\ref{appendix-viscosity}), we know that $V = V_d$. Indeed, $V_d$ is a viscosity solution of the HJB equation and, by definition, $V_d \leq V$. Also, in this case, we have that $V'(0+)=V_d(0+)$ exists and is such that $V'(0+) \leq \beta$ by~\eqref{ineq.boundsV}. In conclusion, if $V(0) = 0$, then $\min\left(V(0), \beta - V'(0+) \right) = 0$. Second, if $V(0)>0$, then to prove $V'(0+) = \beta$, in view of~\eqref{ineq.boundsV}, it suffices to show that
\begin{equation*}
\liminf_{\epsilon \downarrow 0} \frac{V(\epsilon) - V(0)}{\epsilon} \geq \beta .
\end{equation*}
We want to argue by contradiction that for all $\eta>0$ there exists $\epsilon$ such that
$$
V(\epsilon) \leq V(0) + (\beta - \eta) \epsilon .
$$
Fix $\eta>0$ and take $\epsilon < \min(V(0)/\eta,1)$. For $\pi \in \Pi^0$, define $\tau_\epsilon^\pi = \inf\{t\geq 0 : X^\pi_t = \epsilon\}$. By the DPP for $V$ (see Proposition~\ref{DPP-for-Vc}) and Proposition~\ref{prop:Vc_reflected},
\begin{eqnarray*}
  V(0) & = & \sup_{\pi = (\ell,G) \in \Pi^0} \e_0 \left[ \int_0^{T^\pi \wedge \tau_\epsilon^\pi} \mathrm{e}^{-qs}(\ell_s \dd s - \beta \dd G_s) + \mathrm{e}^{- q \tau_\epsilon^\pi} \mathbf{1}_{\{\tau_\epsilon^\pi < T^\pi\}} V(\epsilon) \right] \\
   & \leq & \sup_{\pi = (\ell,G) \in \Pi^0} \e_0 \left[ \int_0^{T^\pi \wedge \tau_\epsilon^\pi} \mathrm{e}^{-qs} \beta (\ell_s \dd s -  \dd G_s) + \mathrm{e}^{- q \tau_\epsilon^\pi} \mathbf{1}_{\{\tau_\epsilon^\pi < T^\pi\}} (V(0) + (\beta - \eta) \epsilon) \right] \\
   & \leq & V(0) - \eta \epsilon + \beta \sup_{\pi = (\ell,G) \in \Pi^0} \e_0 \left[ \int_0^{T^\pi \wedge \tau_\epsilon^\pi} \mathrm{e}^{-qs} (\ell_s \dd s -  \dd G_s) + \mathrm{e}^{- q \tau_\epsilon^\pi}  X^\pi_{\tau_\epsilon^\pi \wedge T^\pi} \right] \\
   &=&  V(0) - \eta \epsilon + \beta \sup_{\pi = (\ell,G) \in \Pi^0} \e_0 \left[ \int_0^{T^\pi \wedge \tau_\epsilon^\pi} \mathrm{e}^{-qs} (\mu - q X^\pi_s ) \dd s\right] \\
   &\leq &  V(0) - \eta \epsilon + \frac{\beta \mu}{q} \sup_{\pi = (\ell,G) \in \Pi^0} \left( 1 - \e_0 \left[ \mathrm{e}^{-q(T^\pi \wedge \tau_\epsilon^\pi)} \right]\right) .
\end{eqnarray*}

Set $\tilde{\pi} = (F(1),\tilde{G}^0) \in \Pi_c$, i.e., the strategy with constant dividend rate $F(1)$ and forced injections at zero, and set $\tilde{\tau}_\epsilon = \inf\{t \geq 0 : X^{\tilde{\pi}}_t = \epsilon\}$. Let $\pi=(\ell,G) \in \Pi^0$. Since $\ell_t \leq F(1)$ for all $t \leq T^\pi \wedge \tau^\pi_\epsilon$, we have that $X^{\tilde{\pi}}_t \leq X^\pi_t$ for all $t \leq T^\pi \wedge \tau^\pi_\epsilon$.

Furthermore, if $\tau^\pi_\epsilon < T^\pi$ then $\tilde{\tau}_\epsilon \geq \tau^\pi_\epsilon$, while if $T^\pi < \tau^\pi_\epsilon$ then $\tilde{\tau}_\epsilon \geq T^\pi$. Therefore, $\tilde{\tau}_\epsilon \geq \tau^\pi_\epsilon \wedge T^\pi$. As a result,
\[
V(0) \leq V(0) - \eta \epsilon + \frac{\beta \mu}{q} \left( 1 - \e_0\left[ \mathrm{e}^{-q\tilde{\tau}_\epsilon} \right] \right) = V(0) - \eta \epsilon + \frac{\beta \mu}{q} \left( 1 - \frac{\tilde{\Zq}(0)}{\tilde{\Zq}(\epsilon)} \right)
\]
in which $\tilde{\Zq}$ is defined as $\Zq$ but with $\mu$ replaced by $\mu-F(1)$. This implies that
\[
\eta \leq \frac{\beta \mu}{q \epsilon} \left( 1 - \tilde{\Zq}(0)/\tilde{\Zq}(\epsilon) \right) \xrightarrow[\epsilon \to 0]{} \frac{\beta \mu}{q} \tilde{\Zq}'(0) = 0 ,
\]
which is a contradiction with $\eta > 0$.
\end{proof}

\subsection{Proof of Theorem~\ref{th.dichotomy}}

Taking into account the solution of the problem without injections in Section~\ref{sec:Vd} and the solution of the problem with forced injections in Section~\ref{sec:Vc}, to complete the proof of Theorem~\ref{th.dichotomy}, we only need to prove that:
\begin{equation}
 V   =  \left\{
  \begin{array}{ll}
    V_d, & \hbox{if $V_d'(0+) \leq \beta$,} \\
    V_c, & \hbox{if $V_d'(0+) > \beta$,}
  \end{array} \label{eqV2}
\right.
\end{equation}
or, equivalently,
\begin{equation}
V  =  \left\{
  \begin{array}{ll}
    V_d, & \hbox{if $V_c(0) < 0$,} \\
    V_c, & \hbox{if $V_c(0) \geq 0.$}
  \end{array}
  \right.  \label{eqV3}
 \end{equation}

Before proceeding with the proofs of~\eqref{eqV2} and~\eqref{eqV3}, recall that $V$, $V_d$ and $V_c$ are all viscosity solutions of the same HJB equation, albeit with different initial conditions. This fact will be used repeatedly, together with the Comparison Principle (see Theorem~\ref{thm:comparison} in Appendix~\ref{appendix-viscosity}), in what follows.

Recall that, by definition of the control problem, we have $V(0)\geq 0$. As an intermediate step, we first prove that
\begin{equation} V  =
 \left\{
  \begin{array}{ll}
    V_d, & \hbox{if $V(0) = 0$,} \\
    V_c, & \hbox{if $V(0) > 0$.}
  \end{array} \label{eqV1}
 \right.
\end{equation}
First, in the proof of Theorem~\ref{thm:viscosity-characterization}, it was already shown  that if $V(0) = 0$ then $V = V_d$. It was also shown  that, if $V(0) > 0$ then $V'(0+)=\beta$. Recall, from Theorem~\ref{thm:viscotity_characterization_Vc}, that $V_c$ is a viscosity supersolution of~\eqref{eq.visco} with initial condition $V_c'(0+)=\beta$. Consequently, since $V$ is a viscosity subsolution of~\eqref{eq.visco} (cf.\ Theorem~\ref{thm:viscosity-characterization}), using the Comparison Principe, we can conclude that $V_c \geq V$ and hence $V = V_c$. This completes the proof of~\eqref{eqV1}.

Now, to prove the dichotomy in~\eqref{eqV2}, we need to show that, if $V_d'(0+) \leq \beta$, then $V_d \geq V$, while if $V_d'(0+) > \beta$, then $V_c \geq V$. We proceed as follows:
\begin{itemize}
\item[(a)] Assume $V_d'(0+) > \beta$. By the dichotomy in~\eqref{eqV1}, if $V(0) = 0$ then $V = V_d$, meaning that $V_d'(0+)=V'(0+) > \beta$. This last inequality is in contradiction with~\eqref{eq.visco2}, the initial condition satisfied by $V$, when $V(0) = 0$. So, if $V_d'(0+) > \beta$ then $V(0)>0$ and, by the dichotomy in~\eqref{eqV1}, we have $V=V_c$.

\item[(b)] Assume $V_d'(0+) \leq \beta$. Again, if $V(0) = 0$, then $V = V_d$. Otherwise, if $V(0) > 0$, then by~\eqref{eq.visco2}, the initial condition satisfied by $V$, we must have $V'(0+)=\beta \geq V_d'(0+)$.  Then, by the Comparison Principle, we have $V_d \geq V$.
\end{itemize}

Similarly, to prove the dichotomy in~\eqref{eqV3}, we proceed as follows:
\begin{itemize}
\item[(a)] Assume $V_c(0) < 0$. Again, if $V(0) = 0$, then $V = V_d$. If $V(0)>0$, then by~\eqref{eqV1} we have $V=V_c$ and thus $V(0)=V_c(0)<0$, which is a contradiction. 

\item[(b)] Assume $V_c(0) \geq 0.$ If $V(0)=0$, then by the Comparison Principle we have $V \leq V_c$. Instead, if $V(0)>0$, then by~\eqref{eqV1} we have $V=V_c$.
\end{itemize}

Note that, using the Comparison Principle, we can conclude that if $V_d'(0+)=\beta$ then $V_c=V_d$, while if $V_c(0)=0$ then $V_d = V_c$. With the above results, this means that if $V_d'(0+)=\beta$ or $V_c(0)=0$, then $V=V_d=V_c$. This concludes the proof of Theorem~\ref{th.dichotomy}.

\section*{Acknowledgements}

Funding in support of this work was provided by three Discovery Grants from the Natural Sciences and Engineering Research Council of Canada (NSERC).

%
%
\bibliographystyle{abbrv}
\bibliography{references_de-finetti}

\appendix

\section{A primer on viscosity solutions and the comparison principle}\label{appendix-viscosity}

The theory of viscosity solutions allows us to characterize a value function as the solution of the HJB equation even though it is not known a priori that this function is differentiable. The definition of a viscosity solution of the HJB equation~\eqref{eq.visco} is as follows:
\begin{definition}\label{def.visco}
A continuous function $u$ is a viscosity supersolution (resp.\ subsolution) of~\eqref{eq.visco} if for all $x \in (0,\infty)$, and for all $\phi \in C^2 [0,\infty)$ for which $\phi(x) = u(x)$ and $u \geq \phi$ (resp.\ $u \leq \phi$), we have
\begin{eqnarray} q \phi(x) - \sup_{0 \leq l \leq F(x)} \mathcal L_l [\phi](x) \geq 0  \mbox{ (resp.\ $\leq 0$)}.\end{eqnarray}  The function $u$ is a viscosity solution of~\eqref{eq.visco} if it is both a supersolution and subsolution of~\eqref{eq.visco}.
\end{definition}

A powerful tool in the theory of viscosity solutions is the comparison principle, which allows the comparison of a supersolution and a subsolution of the HJB equation by only comparing their value (or derivative) at the boundary of the domain. Comparison principles are often used to show uniqueness of the solution given specific initial conditions. Indeed, if $V_1$ and $V_2$ are two solutions of~\eqref{eq.visco} with the same initial condition than they are equal as they are both super and subsolutions. On the other hand, we use the following comparison principle extensively in Section~\ref{sec:dichotomy} to prove the dichotomy result by comparing the value functions $V_c$ and $V_d$ with the value function $V$.

\begin{theorem}[Comparison Principle] \label{thm:comparison}
If $V_1  \in \mathcal W$ (resp.\ $V_2  \in \mathcal W$) is a viscosity supersolution (resp.\ subsolution) of~\eqref{eq.visco} and if one of the following conditions is satisfied:
\begin{enumerate}
\item $V_1(0) \geq V_2(0)$;
\item $V'_1(0+) \leq V'_2(0+)$, assuming both derivatives exist,
\end{enumerate}
then $V_1 \geq V_2$.
\end{theorem}
\begin{proof}
Let $0 < \lambda < 1,$ $A_1 < V_1'(0+)$, $A_2 >0$ and $A_0 > V_1(0)$ large enough that $V_3(x) := A_0 + A_1 x +  A_2 x^2$ is a strict supersolution of (\ref{eq.visco}). It is then clear that $V_1^\lambda := (1-\lambda) V_1 +\lambda V_3$ is a strict viscosity supersolution of (\ref{eq.visco}). Define
$$u_\lambda(x) = V_2(x) -  V_1^\lambda(x).$$ To prove that $V_1 \geq V_2$ it suffices to prove that $u_\lambda \leq 0$ for all $\lambda.$ To prove this, we use a typical argument and assume by contradiction that $M:=\sup_x u_\lambda(x)>0.$ Since $V_1, V_2 \in \mathcal{W}$, we know that $\lim_{x\to \infty}u_\lambda(x) = -\infty.$ Therefore, there exists $x_0 \geq 0$ such that $M = u_\lambda(x_0).$ If $V_1(0) \geq V_2(0)$, then $u_\lambda(0) < 0$, therefore $x_0 >0.$ On the other hand, if Condition (2) of the theorem is satisfied, then $u_\lambda'(0+) = V_2'(0+) - (1-\lambda) V_1'(0+) - \lambda A_1 \geq  \lambda (V_1'(0+) - A_1) >0.$ Consequently, $x_0 >0$.

For $\epsilon>0$, let $\Phi^\epsilon(x,y) = V_2(x) - V_1^\lambda(y) - \phi_\epsilon(x,y)$ with $$ \phi_\epsilon(x,y) = \frac14 |x - x_0|^4 + \frac{1}{2 \epsilon}|x-y|^2.$$ For each $\epsilon>0,$ the function $\Phi^\epsilon$ attains its maximum $M_\epsilon$ at some point $(x_\epsilon,y_\epsilon) \in (0,\infty)^2.$ The rest of the proof is classical and consists in showing that $(x_\epsilon,y_\epsilon)$ converges to $(x_0,x_0)$ as $\epsilon \to 0,$ and $$ \lim_{\epsilon \to 0} \frac{|x_\epsilon - y_\epsilon|}{\epsilon} = 0,$$ in order to apply Theorem 3.2 of \cite{CIL1992} to obtain a contradiction. See for example \cite{CLS2013}, and \cite{pham_1998} for the absolutely continuous control case.
\end{proof}

\section{Auxiliary results for reflected SDEs}\label{AppendixA}

Here are two technical lemmas.

\bigskip

\begin{lemma} \label{lemma.1}
Let $\pi =(\ell,G) \in \Pi$ and consider $\tilde \pi = (\tilde \ell, \tilde G) \in \Pi^0$ such that $\tilde \ell \geq \ell$ and $\tilde G_t = \underline{G}^{\tilde \ell}_{t\wedge T^\pi}$.

If $X^{\tilde \pi}_0 = X^{\pi}_0$, then $T^{\tilde \pi} = T^\pi$ and $X^{\tilde \pi}_t \leq X^{\pi}_t$ for all $t \leq T^\pi.$ In particular,  if $\tilde \ell = \ell$, then $\tilde G_t \leq G_t$ for all $t \leq T^\pi$ and
\[
\int_0^{T^{\tilde \pi}} \mathrm{e}^{-q t} \dd \tilde G_t \leq \int_0^{T^{\pi}} \mathrm{e}^{-q t} \dd G_t .
\]
\end{lemma}
\begin{proof}
Set $Y_t =  X^{\tilde \pi}_t - X^\pi_t$, which is such that $\dd Y_t  = ( \ell_t - \tilde \ell_t) \dd t - \dd G_t +  \dd \tilde G_t$. Note that, for $t\leq T^\pi$, if  $Y_t > 0$ then $X^{\tilde \pi}_t > 0$. As a result, $\mathbf{1}_{\{Y_t > 0\}} \dd \tilde G_t = 0$. Since $Y$ has paths of finite variation and $\tilde G$ has continuous paths, we can use the Change of Variables formula (see Theorem II.31 in \cite{P2005}) and write, for $t<T^\pi$,
\begin{eqnarray*}
 \tfrac12 \max(0,Y_t)^2 &=& \int_0^t \max(0,Y_s)\left(( \ell_s - \tilde \ell_s) \dd s - \dd G_s + \dd \tilde G_s \right) \\ & & \qquad \qquad  + \tfrac12  \sum_{s \leq t} (\max(0,Y_{s-} - (G_s-G_{s-}))^2 - \max(0,Y_{s-})^2)  \\
   &\leq & \int_0^t Y_s \mathbf{1}_{\{Y_s > 0\}} \left(- \dd G_s + \dd \tilde G_s \right) = - \int_0^t Y_s \mathbf{1}_{\{Y_s > 0\}} \dd G_s \leq 0 ,
\end{eqnarray*}
where we used that $G_t \geq G_{t-}$ for all $t\geq 0$. Consequently, we have $Y_t \leq 0$, which means that $X^{\tilde \pi}_t \leq X^{\pi}_t$ for all $t \leq T^\pi$, and hence $T^{\tilde \pi} \leq T^\pi$. But, since by definition $\tilde G_t = \underline{G}^{\tilde \ell}_t$ for all $t\leq T^\pi$, which implies that $X^{\tilde \pi}_t \geq 0 $ for all $t\leq T^\pi$, we also have that $T^{\tilde \pi} \geq T^\pi$.

Now, if $\tilde \ell = \ell$, then it is clear that $X^{\tilde \pi}_t \leq X^{\pi}_t$ implies $\tilde G_t \leq G_t$ for $t\leq T^\pi$. Consequently, by integration by parts, we have
\[
\int_0^{T^\pi} \mathrm{e}^{-qt} \mathrm{d}(\tilde G_t - G_t) = \mathrm{e}^{-q T^\pi} (\tilde G_{T^\pi} - G_{T^\pi}) + q \int_0^{T^\pi} \mathrm{e}^{-qt} (\tilde G_t - G_t) \mathrm{d}t \leq 0 .
\]
\end{proof}

\begin{lemma} \label{lemma.Skorohod}
For $b>0$ and $x \geq 0$, there exist an adapted nonnegative process $\underline{X}^b$ such that
\[
 \dd  \underline{X}^b_t = \left(\mu - F(\underline{X}^b_t) \mathbf{1}_{\{\underline{X}^b_t \geq b \}} \right)  \dd t + \sigma  \dd  B_t +  \dd  G_t, \quad \underline{X}^b_0=x,
\]
and an adapted nondecreasing process $G$ such that
\[
G_t = \int_0^t \mathbf{1}_{\{\underline{X}^b_s = 0 \}}  \dd  G_s .
\]
\end{lemma}
\begin{proof}
Set $H(y):=\mu - F(y) \mathbf{1}_{\{y \geq b \}}$. Recall that $F:[0,\infty) \to [0,\infty)$ is a continuously differentiable, nondecreasing and concave function, therefore a Lipschitz function. Recall also that $X_t = x + \mu t + \sigma B_t$.

Set $T_0 = 0$ and let $X^0$ be the solution to the following SDE on $[0,\infty)$:
$$
\dd X^0_t = H(X^0_t) \dd t + \sigma \dd B_t, \quad X^0_{0} = x.
$$
This process is indeed the same as the controlled process associated with a mean-reverting strategy at level $b$, but we use a different notation only for the purpose of this proof.

Set $G^0 \equiv 0$ and $T_1 = \inf\{t\geq T_0 : X^0_t = b\}$. By induction, for $n\geq 1$:
\begin{itemize}
\item if $n$ is odd: for $t \geq T_n$, define $G^n_t = - \min_{T_n \leq s \leq t} ((X_s - X_{T_n}) \wedge 0)$ and $X^n_t = X_t - X_{T_n} + G^n_t$. Set $T_{n+1} = \inf\{t\geq T_n : X^n_t = b\}$.
\item if $n$ is even: for $t \geq T_n$, define $X^n$ as the solution of the following SDE:
$$
\dd X^n_t = H(X^n_t) \dd t + \sigma \dd B_t, \quad X^n_{T_n} = b .
$$
Also, set $G^n \equiv 0$ and $T_{n+1} = \inf\{t\geq T_n : X^n_t = 0\}$.
\end{itemize}
Note that, for $n \geq 1$ odd, the pair $(X^n,G^n)$ is the solution to the following Skorohod problem on $[T_n,T_{n+1})$:
\begin{align*}
\dd X^n_t &= \mu \dd t + \sigma \dd B_t + \dd G^n_t, \quad X^n_{T_n} = 0,\\
G^n_t &= \int_{T_n}^t \mathbf{1}_{\{X^n_s = 0\}} \dd G^n_s .
\end{align*}

Finally, define
\[
\underline{X}^b_t = \sum_{n \geq 0} X^n_t \mathbf{1}_{\{T_n \leq t < T_{n+1}\}} \quad \text{and} \quad G_t = \sum_{n \geq 0} G^n_t \mathbf{1}_{\{t \leq T_{n+1}\}} .
\]
The pair $(\underline{X}^b,G)$ satisfies the conditions in the statement of the lemma. The details are left to the reader.
\end{proof}

\section{Proof of Proposition~\ref{prop.C1}}\label{preuve-prop.C1}

\bigskip

Since $V_c$ is a nondecreasing, continuous and concave function, we know that its left and right derivatives $V_c^{',-}, V_c^{',+}$ exist.
\begin{itemize}
\item[(a)] Assume there is $x_0>0$ and $d > 0$ such that $V_c^{',-}(x_0)< d < V_c^{',+}(x_0).$ For $M \in \mathbb{R}$, define $\phi_M(x) = V_c(x_0) + d (x-x_0) + \tfrac{1}{2} M (x-x_0)^2.$ Then $x_0$ is a local minimum of $V_c - \phi_M$ when $M>0$. By the supersolution property of $V_c$, we must have that $$ q \phi_M(x_0) - \sup_{0 \leq l \leq F(x_0)} \mathcal L_l[\phi_M](x_0) = q \phi_M(x_0) - \sup_{0\leq l \leq F(x_0)} \left[(\mu-l) d - l \right] - \tfrac12 \sigma^2 M \geq 0.$$ Taking $M>0$ large enough yields a contradiction.

\item[(b)] Assume there is $x_0>0$, $d > 0$ such that $V_c^{',+}(x_0)< d < V_c^{',-}(x_0).$  Then $x_0$ is a local maximum of $V_c - \phi_M$, when $M < 0.$ By the subsolution property of $V_c$, we must have that
$$ q \phi_M(x_0) - \sup_{0 \leq l \leq F(x_0)} \mathcal L_l[\phi_M](x_0) =  q \phi_M(x_0) - \sup_{0 \leq l \leq F(x_0)} \left[ (\mu-l) d - l \right] - \tfrac12 \sigma^2 M \leq 0.$$ Taking $M \to - \infty$ leads to a contradiction.
\end{itemize}
Since, by Proposition~\ref{prop:Vc-derivative-at-zero}, we already know that $V'_c(0) = \beta$, this concludes the proof.

\section{A first-passage identity} \label{AppendixC}

Here is a solution to a first-passage problem for the refracted diffusion process $X^b$. We use the notation introduced in Section~\ref{sec:notation}.

\begin{lemma}\label{Lemma:fluctuation-identity}
For a fixed $b \geq 0$, we have
\[
\e_x \left[ \mathrm{e}^{-q \tau_0^{X^b}} \right] =
\begin{cases}
\Zq (x) + \Wq(x) \left[ \frac{\Zq (b) \Hqprime (b) - \Zqprime (b) \Hq (b)}{\Wqprime (b) \Hq (b) - \Wq (b) \Hqprime (b)} \right] & \text{if $0 \leq x \leq b$,}\\
\Hq(x) \left[ \frac{\Zq (b) \Wqprime (b) - \Zqprime (b) \Wq (b)}{\Wqprime (b) \Hq (b) - \Wq (b) \Hqprime (b)} \right] & \text{if $x \geq b$.}
\end{cases}
\]
\end{lemma}
\begin{proof}
Define $h(x) = \e_x \left[ \mathrm{e}^{-q \tau_0^{X^b}} \right]$. For $0 \leq x \leq b$, the process $\left\lbrace X^b_t , 0 \leq t \leq \tau_0^{X^b} \wedge \tau_b^{X^b} \right\rbrace$ has the same distribution (with respect to $\p_x$) as $\left\lbrace X_t , 0 \leq t \leq \tau_0^X \wedge \tau_b^X \right\rbrace$. Consequently, using the strong Markov property and the almost-sure continuity of trajectories, we can write
\begin{align*}
h(x) &= \e_x \left[ \mathrm{e}^{-q \tau_0^{X^b}} \ind_{\{\tau_0^{X^b}<\tau_b^{X^b}\}} \right] + \e_x \left[ \mathrm{e}^{-q \tau_b^{X^b}} \ind_{\{\tau_b^{X^b}<\tau_0^{X^b}\}} \right] h(b) \\
&= \e_x \left[ \mathrm{e}^{-q \tau_0^X} \ind_{\{\tau_0^X<\tau_b^X\}} \right] + \e_x \left[ \mathrm{e}^{-q \tau_b^X} \ind_{\{\tau_b^X<\tau_0^X\}} \right] h(b) \\
&= \left( \Zq(x) - \frac{\Zq(b)}{\Wq(b)} \Wq(x) \right) + \frac{\Wq(x)}{\Wq(b)} h(b) ,
\end{align*}
where, in the last equality, we used two well-known fluctuation identities for a Brownian motion with drift; see, e.g., \cite{APP2007}.

For $x \geq b$, the process $\left\lbrace X^b_t , 0 \leq t \leq \tau_b^{X^b} \right\rbrace$ has the same distribution (with respect to $\p_x$) as $\left\lbrace X^0_t , 0 \leq t \leq \tau_b^{X^0} \right\rbrace$. Consequently, using the strong Markov property and the almost-sure continuity of trajectories, we can write
\[
h(x) = \e_x \left[ \mathrm{e}^{-q \tau_b^{X^0}} \right] h(b) = \frac{\Hq(x)}{\Hq(b)} h(b) .
\]
where, in the last equality, we used a known fluctuation identity for the process $X^0$; see, e.g., \cite{LR2023}.

To compute $h(b)$, we will use the following approximating sequence. This methodology has been used several times in the literature; see, e.g., \cite{LR2023}. For $n \geq 1$ and initial value $x<b$, let $X^{b,n}$ have the same dynamics as that of $X$, until it reaches level $b$. Then, its dynamics change to that of $X^0$, until it goes back to level $b-1/n$, in which case its dynamics goes back to that of $X$. This procedure is repeated until level $0$ is attained.

Define $h_n(x)=\e_x \left[ \mathrm{e}^{-q \tau_0^n} \right]$, where $\tau^n_0 = \inf \left\lbrace t > 0 \colon X^{b,n}_t = 0 \right\rbrace$. Using arguments as above (such as the Markov property and almost-sure continuity of trajectories), we can write
\[
h_n(b-1/n) = \e_{b-1/n} \left[ \mathrm{e}^{-q \tau_0^X} \ind_{\{\tau_0^X<\tau_b^X\}} \right] + \e_{b-1/n} \left[ \mathrm{e}^{-q \tau_b^X} \ind_{\{\tau_b^X<\tau_0^X\}} \right] h_n(b) ,
\]
where $h_n(b) = \e_b \left[ \mathrm{e}^{-q \tau_{b-1/n}^{X^0}} \right] h_n(b-1/n)$. Replacing, solving for $h_n(b-1/n)$ and then computing the expectations, we get
\[
h_n(b-1/n) = \frac{\Zq(b-1/n) - \frac{\Zq(b)}{\Wq(b)} \Wq(b-1/n)}{1 - \frac{\Wq(b-1/n)}{\Wq(b)} \frac{\Hq(b)}{\Hq(b-1/n)}} .
\]
Taking the limit, we further get
\[
h(b) = \frac{-\Zqprime(b) + \frac{\Zq(b)}{\Wq(b)} \Wqprime(b)}{\frac{\Wqprime(b)}{\Wq(b)} - \frac{\Hqprime(b)}{\Hq(b)}}
\]
and the result follows.
\end{proof}

\end{document}